\documentclass[pdflatex]{article}
\usepackage{algorithm}%
\usepackage{algorithmicx}%
\usepackage[inner=2cm,outer=1cm]{geometry}
\usepackage{amsthm,amsmath}
\usepackage[utf8]{inputenc}

\usepackage{graphicx} 
\usepackage{xcolor,color}
\usepackage{xurl}
\usepackage{doi}
\usepackage[numbers,compress,sort]{natbib} 
\usepackage{amsmath,amssymb} 
\usepackage{afterpage,rotating,pdflscape} 

\newcommand{\ignore}[1]{}
\usepackage{subfigure}
\usepackage{marginnote}

\numberwithin{figure}{section}
\numberwithin{table}{section}

\newcommand{\bpi}{\boldsymbol\pi}
\newcommand{\blambda}{\boldsymbol\lambda}
\DeclareSymbolFont{bbold}{U}{bbold}{m}{n}
\DeclareSymbolFontAlphabet{\mathbbold}{bbold}
\newcommand{\ind}[1]{\mathbbold{1}_{\left[ #1 \right]}}

\usepackage{color} 
\usepackage{graphics,graphicx} 
\graphicspath{ {./Figures/}}

\newcommand{\Reals}{\mathbb{R}} 
\newcommand{\Pbb}{\mathbb{P}}
\newcommand{\Ebb}{\mathbb{E}}
\newcommand{\Vbb}{\mathbb{V}}

\newcommand{\defined}{:=} 

\newcommand{\zerob}{\mathbf{0}} 
\newcommand{\oneb}{\mathbf{1}} 
\usepackage{bm} 

\newcommand{\ab}{\bm{a}}

\newcommand{\cb}{\bm{c}}

\newcommand{\gb}{\bm{g}}

\newcommand{\sba}{\bm{s}}

\newcommand{\vb}{\bm{v}}

\newcommand{\xb}{\bm{x}}

\newcommand{\yb}{\bm{y}}

\newcommand{\cB}{\mathcal{B}}

\usepackage{lipsum}

\usepackage{comment}

\usepackage{ulem}
\usepackage[english]{babel}
\usepackage[export]{adjustbox}
\usepackage{cleveref}
\newtheorem{proposition}{Proposition}
\crefname{prop}{Proposition}{Propositions}
\newtheorem{assumption}{Assumption}
\crefname{ass}{Assumption}{Assumptions}
\newtheorem{theorem}{Theorem}
\crefname{thm}{Theorem}{Theorems}

\newtheorem{definition}{Definition}[theorem]
\usepackage{float}
\usepackage[ruled,vlined,linesnumbered,algo2e]{algorithm2e}
\SetKwInOut{Input}{input}
\SetKwInOut{Output}{output}
\SetKw{Break}{break}
\SetAlgoNlRelativeSize{-5}
\SetKw{Continue}{continue}
\DontPrintSemicolon 
\usepackage{url}
\hypersetup{breaklinks=true} 

\usepackage{orcidlink}

\crefname{alg}{Algorithm}{Algorithms}
\crefname{equation}{}{}

\begin{document}
\title{Stochastic Average Model Methods}
\author{Matt Menickelly, Stefan M. Wild}
\maketitle




\abstract{
 We consider the solution of finite-sum minimization problems, such as those appearing in nonlinear least-squares or general empirical risk minimization problems. 
 We are motivated by problems in which the summand functions are computationally expensive and evaluating all summands on every iteration of an optimization method may be undesirable. We present the idea of stochastic average model (\texttt{SAM}) methods, inspired by stochastic average gradient 
 methods. \texttt{SAM} methods sample component functions on each iteration of a trust-region method according to a discrete probability distribution on component functions; the distribution is designed to minimize an upper bound on the variance of the resulting stochastic model. We present promising numerical results concerning an implemented variant extending the derivative-free model-based trust-region solver \texttt{POUNDERS}, which we name \texttt{SAM-POUNDERS}.
}




\section{Introduction}
We consider the minimization of an objective comprising a sum of component functions,
\begin{equation}
\label{eq:sum}
f(\xb) = \sum_{i=1}^p F_i(\xb),
\end{equation}
for parameters $\xb\in\Reals^n$.

The minimization problem \cref{eq:sum} is ubiquitous in computational optimization, with applications across computational science, engineering, and industry.
Statistical estimation problems, such as those appearing in empirical risk minimization, 
can be described in the form of \cref{eq:sum}.
In such a setting, one lets $\xb$ denote statistical model parameters and lets each $F_i$ denote a likelihood function associated with one of $p$ empirical data points. 
We refer to the general problem in \cref{eq:sum} as finite-sum minimization.  The literature for solving this problem when $p$ is large is now massive, mainly due to the use of empirical risk minimization in supervised learning. 
Most such methods are based fundamentally on the stochastic gradient iteration
\cite{RobbinsMonro1951}, which works (in the finite-sum minimization setting) by iteratively approximating a gradient $\nabla f(\xb^k)$ by $\nabla F_i(\xb^k)$ for a randomly chosen $i$ and updating $\xb^{k+1}\gets \xb^k - \gamma_k \nabla F_i(\xb^k)$, for some $\gamma_k > 0$.   
When gradients are unavailable or prohibitively expensive to compute---casting the problem of \cref{eq:sum} as one of derivative-free optimization---gradient-free versions of the stochastic iteration are also plentiful.
Such methods typically replace the stochastic gradient approximation by a finite-difference estimation of the stochastic gradient; this idea dates back to at least 1952 \cite{KieferWolfowitz}, one year after the stochastic gradient iteration was proposed in \cite{RobbinsMonro1951}. 

The setting in which methods based on the stochastic gradient iteration are appropriate are typically marked by several characteristics: 
\begin{enumerate} 
\item \emph{Accuracy} (as measured in terms of an optimality gap) is \emph{not critically important}, and only coarse estimates of the solution to \cref{eq:sum} are required.
\item The number of component functions, $p$, in \cref{eq:sum} is fairly \emph{large}, so that computing (or estimating) $\nabla F_i(\xb)$ is significantly less expensive than computing $\nabla f(\xb)$. 
\item The computation of $\nabla F_i(\xb)$ (or $F_i(\xb)$) is fairly \emph{inexpensive}, typically requiring a number of arithmetic operations linear in $n$.  
\end{enumerate} 

In this paper we are concerned with settings where these assumptions are not satisfied. 
In particular, and in contrast to each of the three points above, we make the following assumptions. 
\begin{enumerate}
\item Problems \emph{must be solved to a particular accuracy} to provide reliable results and model calibrations. 
\item Data is scarce and expensive to obtain, meaning $p$ is \emph{not necessarily large}, and thus---to avoid overparameterization---$n$ is likely not large, either. 
\item  \emph{Computation of $F_i(\xb)$ will be the dominant cost} of any optimization method.
\end{enumerate}

As a  motivating example for these assumptions, 
we refer to the problem of nuclear model calibration in computational science. 
In such problems, each $F_i$ in \cref{eq:sum} is a likelihood term that fits an observable derived from a model of a nucleus parameterized by $\xb$ to empirical data.
The computation of the observable, however, involves a time-intensive computer code. 
The application of a derivative-free trust-region method, \texttt{POUNDERS}, to such problems when the likelihood function is expressed as a least-squares minimization problem, is discussed in \cite{SWCHAP14}. 
More recently, Bollapragada  et al.~\cite{Bollapragada_2020} compared the performance of various derivative-free methods on problems of nuclear model calibration and arrived at some conclusions that partially inspired the present paper. 
Although we are particularly interested in and motivated by derivative-free optimization in this paper, we remark that many of the concerns outlined above also apply to expensive \emph{derivative-based} model calibration; see, for example, \cite{Bouhlel2019}. 

We comment briefly on these differences in problem settings.
For the  issue of accuracy, it is well known that the standard stochastic gradient iteration with a constant step size $\gamma_k$ can  converge (in expectation) only to a particular level of accuracy determined by the variance of the stochastic gradient estimator. 
More formally, by arguments that are now essentially folklore (see, e.g., \cite{Ghadimi2013} or \cite[Section 4.3]{BottCurtNoce16}), given a Lipschitz constant for $\nabla f(\xb)$, $L$, a uniform second moment bound over all $k$, $\mathbb{E}_i\left[\left\|\nabla F_i(\xb^k)\right\|^2\right]\leq M$, and a lower bound on the objective function value, $f_*$, one can demonstrate that for a whole number of iterations $K>0$, the stochastic gradient iteration with step size chosen sufficiently small achieves 
$$\displaystyle\frac{1}{K}\displaystyle\sum_{k=1}^K \mathbb{E}\left[\left\|\nabla f(\xb^k)\right\|^2\right] \leq \displaystyle\frac{L \left( f(\xb^0) - f_*\right)}{K} + \frac{M}{2},$$
where the expectation is taken with respect to the $\sigma$-algebra generated by the random draws of $i$.\footnote{This is obviously a result appropriate for a general nonconvex setting. Stronger results can be proven when additional regularity assumptions, typically strong convexity,  are imposed on $f$. However, because we are motivated by problems where convexity typically should not be assumed, we choose to state this folklore result. We also note that, even in convex settings, stochastic gradient descent with a fixed step size will still involve an irreducible error term dependent on stochastic gradient variance.}
Notably, regardless of how many iterations $K$ of the stochastic gradient method are performed, there is an unavoidable upper bound of $\frac{M}{2}$ on the optimality gap. 

Such a variance-dependent gap can be eliminated by using a variety of  \emph{variance reduction} techniques. 
The simplest such technique entails using a sequence of step sizes that decay to 0 at a sublinear rate. 
Choosing such a decaying step size schedule in practice, however, is known to be difficult. 
More empirically satisfying methods of variance reduction include methods such as stochastic average gradient (\texttt{SAG}) methods \cite{Schmidt2013, roux2012stochastic}, which maintain a running memory of the most recently evaluated $\nabla F_i(\xb^k)$ and reuse those stale gradients when forming an estimator of the full gradient $\nabla f(\xb^k)$. 
Although convergence results can be proven about \texttt{SAG} \cite{Schmidt2013}, such a gradient approximation scheme naturally leads to a biased gradient estimate. 
The algorithm \texttt{SAGA}\footnote{The additional ``A" in \texttt{SAGA} ostensibly stands for ``am{\'e}lior{\'e}", or ``ameliorated"} employs so-called control variates to correct this bias \cite{DefazioBL14}. 
We note that the method presented in this paper is inspired by \texttt{SAG} and \texttt{SAGA}, and even more closely resembles \cite{gower2018tracking}.
Additional variance reduction techniques include ``semi-stochastic" gradient methods, which occasionally, according to an outer loop schedule, compute a full gradient $\nabla f(\xb^k)$; most such methods are inspired by \texttt{SVRG} \cite{Johnson2013, zhang2013linear}.

For the second and third issues, the prevalence and dominance of stochastic gradient methods is empirically undeniable in the setting of supervised learning with big data, where $p$ is large and the component (loss) functions $F_i$ are typically simple functions (e.g., a logistic loss function) of the data points.
In an acclaimed paper \cite{bottou2007tradeoffs}, the preference for stochastic gradient descent over gradient descent in the typical big data setting is more rigorously demonstrated,\footnote{Once again, the results in \cite{bottou2007tradeoffs} are proven under convexity assumptions, but one can see how their conclusions concerning time-to-solution are also valid for nonconvex but inexpensive and large-scale learning.}
illustrating the trade-off in worst-case time-to-solution in terms as a function of desired optimality gap and $p$, among other important quantities.

As mentioned, however, in our setting $p$ is relatively small, and the loss functions $F_i$ are far from computationally simple. 
In fact, the gradients $\nabla F_i$ are often unavailable, necessitating derivative-free methods---or if the gradients $\nabla F_i$ are available, their computation requires algorithmic differentiation (AD) techniques, which are computationally even more expensive than function evaluations of $F_i(\xb)$ and require the human effort of AD experts for many applications. 
In the derivative-free setting, prior work has paid particular attention to the case where each $F_i(\xb)$ is a composition of a square function with a more complicated function, that is, the setting of least-squares minimization \cite{Zhangdfo10, Zhang2012, SWCHAP14, Cartis2017}. 
However, these works do not employ any form of randomization.
Although the general technique we propose in this paper is applicable to a much broader class of finite-sum minimization \cref{eq:sum}, we will demonstrate the use of our technique by extending \texttt{POUNDERS} \cite{SWCHAP14}, which was developed for derivative-free least-squares minimization. 

\section{Stochastic Average Model Methods}
We impose the following assumption on $f$ throughout this manuscript.
\begin{assumption}
\label{ass:f}
Each $F_i$ has a Lipschitz continuous gradient $\nabla F_i$ with constant $L_i$
(and hence $f$ has a Lipschitz continuous gradient).
Additionally, each $F_i$ (and hence $f$) is bounded below. 
\end{assumption} 

For each $F_i$ we employ a component model $m_i:\Reals^n \rightarrow \Reals$.
Each component model $m_i$ is associated with a dynamically updated 
\emph{center point} $\cb^k_i \in \Reals^n$, and we express the model value at a point $\xb \in \Reals^n$ 
as
$m_i(\xb; \cb^k_i)$. 
We use this notation in order to never lose sight of a component model's center point.

We refer to our model of $f$ in \cref{eq:sum} as the 
\emph{average model}, 
\begin{equation*}
\bar{m}^k(\xb) \defined \displaystyle\sum_{i=1}^p m_i(\xb; \cb^k_i),
\end{equation*}
and distinguish it from the model
\begin{equation*}
m^k(\xb) \defined \displaystyle\sum_{i=1}^p m_i(\xb; \xb^k),
\end{equation*}
for which all $p$ component models employ the common center $\cb^k_i = \xb^k$.

The name ``average model" reflects its analogue,
the average gradient, employed in \texttt{SAG} methods~\cite{roux2012stochastic, Schmidt2013}. 
Given a fixed batch size $b$, on iteration $k$ our method selects a subset $I^k\subseteq\{1,\dots,p\}$ of size $\left|I^k\right|=b$ and updates $\cb^k_i$ to the current point $\xb^k$ for all $i\in I^k$. 
The update of $\cb^k_i$ in turn results in an update to the component models $\{m_i: \, i\in I^k\}$. 
In this paper we select $I^k$ in a randomized fashion; we denote the probability of the event that $i\in I^k$ by
$$\pi^k_i \defined \Pbb\left[i\in I^k\right].$$
We remark that such randomized selections of batches in an optimization have been studied in the past and have been referred to as arbitrary sampling \cite{CsibaRichtarik2018,HanzelyRichtarik2019,HorvathRichtarik2019,RichtarikTakac2016};
that body of work motivated the ideas presented here.
After updating the component models $\{m_i(\xb;\cb^k_i): i\in I^k\}$ to $\{m_i(\xb;\xb^k): i\in I^k\}$, 
we employ the \emph{ameliorated model}
\begin{equation}\label{eq:saga_model}
\hat{m}_{I^k}(\xb) \defined  \displaystyle\sum_{i \in I^k} \frac{m_i(\xb;\xb^k)-m_i(\xb;\cb_i^{k-1})}{\pi^k_i} + \bar{m}^{k-1}(\xb),
\end{equation}
recalling that $\bar{m}_{k-1}(\xb)$ is the previous iteration's average model.
In order for \cref{eq:saga_model} to be well defined, we require that $\pi^k_i>0$; that is, we require the probability of sampling the $i$th component function in the $k$th iteration to always be nonzero. 
We remark that the ameliorated model \cref{eq:saga_model} is obviously related to the \texttt{SAGA} model \cite{DefazioBL14}, which is an unbiased correction to the \texttt{SAG} model.
We record precisely what is meant by unbiased correction in the following proposition, and we
stress that the statement of the proposition is effectively independent of the particular randomized selection of $I^k$. 
\begin{proposition}\label{prop:unbiased}
For all samplings defined by $\pi_i^k>0$, $i=1, \ldots,p$, and for all $\xb \in \Reals^n$, the ameliorated model in \cref{eq:saga_model} satisfies $\displaystyle\Ebb_{I^k}\left[ \hat{m}_{I^k}(\xb)\right] = m^k(\xb).$
\end{proposition}

\begin{proof}
$$\begin{array}{rl}
\displaystyle\Ebb_{I^k}\left[ \hat{m}_{I^k}(\xb)\right] = & 
\Ebb_{I^k}\left[\displaystyle\sum_{i=1}^p \frac{m_i(\xb;\xb^k)-m_i(\xb;\cb_i^{k-1})}{\pi^k_i}  \ind{i\in I^k} \right] + \bar{m}_{k-1}(\xb)\\
= & \displaystyle\sum_{i=1}^p \left(m_i(\xb;\xb^k)-m_i(\xb;\cb_i^{k-1})\right) + \bar{m}_{k-1}(\xb)\\
= & \displaystyle\sum_{i=1}^p \left(m_i(\xb;\xb^k)-m_i(\xb;\cb_i^{k-1}) + m_i(\xb;\cb_i^{k-1})\right)\\
= & \displaystyle\sum_{i=1}^p m_i(\xb;\xb^k) = m^k(\xb). 
\end{array}
$$
\end{proof}

To make the currently abstract notions of component model centers and model updates more immediately concrete, 
we initially focus on one particular class of component models; 
in \Cref{eq:beyondfo}, we introduce and discuss three additional classes of component models.
Our first class of component models is first-order (i.e., gradient-based) models of the form
\begin{equation}\label{eq:first_order}
m_i(\xb;\cb^k_i) = F_i(\cb^k_i) + \nabla F_i(\cb^k_i)^\top (\xb-\cb_i^k). \tag{FO}
\end{equation}
Thus, on any iteration $k$ the component models $\{m_i: \, i\in I^k\}$ are updated to 
\begin{equation*}
m_i(\xb;\xb^k) = F_i(\xb^k) + \nabla F_i(\xb^k)^\top (\xb-\xb^k),
\end{equation*}
which entails an additional pair of function and gradient evaluations $(F_i(\xb^k), \, \nabla F_i(\xb^k))$ for each $i\in I^k$. 
The result of \Cref{prop:unbiased} then guarantees that when using the component models \cref{eq:first_order},
\begin{equation*}
\mathbb{E}_{I^k}\left[\hat{m}_{I^k}(\xb)\right] = \displaystyle\sum_{i=1}^p \left(F_i(\xb^k) + \nabla F_i(\xb^k)^\top  (\xb-\xb^k)\right) = f(\xb^k) + \nabla f(\xb^k)^\top (\xb-\xb^k).
\end{equation*}
That is, in expectation over the draw of $I^k$, the ameliorated model recovers the first-order model of \cref{eq:sum} centered at $\xb^k$. 

We will suggest a specific set of probabilities $\{\pi^k_i\}$ later in \Cref{sec:estimator_properties}; but for now, given arbitrary parameters $0 < \pi^k_i \leq 1$,
we can completely describe our average model method in \Cref{alg:dfotr}. 

\begin{algorithm}[h!]
\caption{Average Model Trust-Region Method}
\label{alg:dfotr}
\textbf{(Initialization)} Choose initial point $\xb^0\in\Reals^n$ and initial trust-region radius $\Delta_0\in (0,\Delta_{\max})$ 
with $\Delta_{\max}>0$. 
Choose algorithmic constants $\gamma > 1, \eta_1\in(0,1), \eta_2 > 0$.  \\
Set $\cb^0_i \gets \xb^0$ for $i=1,\dots,p$. \\
Construct each initial model $m_i(\cdot;\cb^0_i)$.\\
$k\gets 1$.\\
\For{$k=1,2,\dots$}
{
\textbf{(Choose a subset)} (Randomly) generate subset $I^k$ and corresponding probability parameters $\{\pi^{k,I^k}_i\}$. \label{line:subset_selection}\\
\textbf{(Update models)} For each $i\in I^k$, $\cb^k_i\gets \xb^k$. For each $i\not\in I^k$, $\cb^k_i\gets \cb^{k-1}_i$.\\
\textbf{(Compute step)} Approximately solve the trust-region subproblem with the average model, $\sba^k\gets \arg\displaystyle\min_{\sba:\|\sba\|\leq\Delta_k} \hat{m}_{I^k}(\xb^k+\sba)$.\\
\textbf{(Choose a second subset)} (Randomly) generate subset $J^k$ and corresponding probability parameters $\{\pi^{k,J^k}_i\}$  \label{line:second_subset_selection}. \\
\textbf{(Compute estimates)} Compute estimates $\hat{m}_{J^k}(\xb^k), \hat{m}_{J^k}(\xb^k+\sba^k)$ of $f(\xb^k)$, $f(\xb^k+\sba^k)$, respectively. This entails computing any previously uncomputed $f_j(\xb^k)$ and $f_j(\xb^k+\sba)$ for $j\in J^k$. \label{line:compute_estimates} \\
\textbf{(Determine acceptance)} Compute $\rho_k\gets \displaystyle\frac{\hat{m}_{J^k}(\xb^k)-\hat{m}_{J^k}(\xb^k+\sba^k)}{\hat{m}_{I^k}(\xb^k)-\hat{m}_{I^k}(\xb^k+\sba^k)}$. \\
\textbf{(Accept point)} \If{$\rho_k\geq\eta_1$ and $\Delta_k \leq \eta_2\|\nabla \hat{m}_{I^k}(\xb^k)\|$}{
$\xb^{k+1}\gets \xb^k + \sba^k$.}
\Else
{$\xb^{k+1}\gets\xb^k$.}
\textbf{(Trust-region adjustment)} \If{$\rho_k\geq\eta_1$ and $\Delta_k \leq \eta_2\|\nabla \hat{m}_{I^k}(\xb^k)\|$ \label{line:tr_adjust}}{
$\Delta_{k+1}\gets \min\{\gamma\Delta_k,\Delta_{\max}\}.$}
\Else
{
$\Delta_{k+1}\gets\gamma^{-1}\Delta_k$.
}
}
\end{algorithm} 

\Cref{alg:dfotr} resembles a standard derivative-free model-based trust-region method. 
More specifically, \Cref{alg:dfotr} is a variant of the \texttt{STORM} method introduced in \cite{Chen2017,BCMS2018}, in that random models and random estimates of the objective (dictated by the random variables $I^k$ and $J^k$, respectively) are employed. We will examine this connection more closely in \Cref{sec:storm}.
At the start of each iteration, $I^k$ is randomly generated according to the discrete probability distributions $\{\pi_i^{k,I^k}\}$ with support $\{1,\dots,p\}$.
Next  we  compute a random model---specifically, an ameliorated model of the form \cref{eq:saga_model}---defined via the random variable $I^k$. 
We then (approximately\footnote{By approximately, we mean the solution to the trust-region subproblem should satisfy a fraction of Cauchy decrease; see, e.g., \cite{Trmbook}.}) solve a trust-region subproblem, minimizing the ameliorated model $\hat{m}_{I^k}$ over a trust region of radius $\Delta_k$ to obtain a trial step $\sba^k$. 
We then compute random estimates of the objective function value at the incumbent and trial points ($f(\xb^k)$ and $f(\xb^k+\sba^k)$, respectively) by constructing a second ameliorated model $\hat{m}_{J^k}$ and then evaluating $\hat{m}_{J^k}(0)$ and $\hat{m}_{J^k}(\sba^k)$. 
If the decrease as measured by $\hat{m}_{J^k}$ is sufficiently large compared with the decrease predicted from the solution of the trust-region subproblem, then, as in a standard trust-region method, the trial step is set as the incumbent step of the next iteration, and the trust-region radius is increased.
Otherwise, the incumbent step is unchanged, and the trust-region radius is decreased.

\section{Ameliorated Models $\hat{m}_{I^k}$ and $\hat{m}_{J^k}$}\label{sec:estimator_properties}
Throughout this section we will continue to assume that models are of the form \cref{eq:first_order}, for the sake of introducing ideas clearly. 

\subsection{Variance of $\hat{m}_{I^k}$} 
Having demonstrated that $\hat{m}_{I^k}(\xb)$ is a pointwise unbiased estimator of an unknown (but---at least in the case of \cref{eq:first_order}---meaningful) model in \Cref{prop:unbiased}, it is reasonable to question what the variance of this estimator is. 
Denote the probability that both indices $i,j\in I^k$ by
$$\pi_{ij}^k = \Pbb\left[i,j\in I^k\right].$$

\begin{proposition}\label{prop:var}
The variance of $\hat{m}_{I^k}(\xb)$, for any $\xb\in\Reals^n$, is
\begin{equation}\label{eq:var}\displaystyle\Vbb_{I^k}\left[\hat{m}_{I^k}(\xb)\right] = \sum_{(i,j)\in [\![ p ]\!] \times [\![ p ]\!]}\left(\frac{\pi^k_{ij}}{\pi^k_i \pi^k_j} - 1\right)d^k_id^k_j,
 \end{equation}
where we denote
$d^k_i = m_i(\xb;\xb^k) - m_i(\xb;\cb^{k-1}_i)$
and abbreviate $\{1,\dots,p\}$ as $[\![ p ]\!]$.
\end{proposition}

\begin{proof}
Let $m^k(\xb)$ denote the expectation over $I^k$ of $\hat{m}_{I^k}(\xb)$. 
Then,
$$\begin{array}{rl}
\displaystyle\Vbb_{I^k}\left[\hat{m}_{I^k}(\xb)\right]
= & \Ebb_{I^k}\left[\left(\hat{m}_{I^k}(\xb) - m^k(\xb)\right)^2\right]\\
= & \Ebb_{I^k}\left[\left(\displaystyle\sum_{i=1}^p\displaystyle\frac{\ind{i\in I_k} d^k_i}{\pi^k_i} + \bar{m}_{k-1}(\xb)\right)^2\right] - m^k(\xb)^2\\
= & 
\mathbb{E}_{I^k}\left[ \displaystyle\sum_{(i,j)\in [\![ p ]\!] \times [\![ p ]\!]} \displaystyle\frac{d^k_i d^k_j}{\pi^k_i \pi_j^k} \ind{(i,j)\in I^k} + \bar{m}_{k-1}(\xb)^2 +2\bar{m}_{k-1}(\xb)\displaystyle\sum_{i =1}^p \frac{d^k_i}{\pi^k_i}\ind{i\in I^k} \right] \\
& - m^k(\xb)^2\\
= & \mathbb{E}_{I^k}\left[ \displaystyle\sum_{(i,j)\in[\![ p ]\!] \times [\![ p ]\!]} \displaystyle\frac{d^k_i d^k_j}{\pi^k_i \pi_j^k} \ind{(i,j)\in I^k}\right] + 2\bar{m}_{k-1}(\xb)\mathbb{E}_{I^k}\left[\displaystyle\sum_{i =1}^p \frac{d^k_i}{\pi^k_i}\ind{i\in I^k} \right] \\
& + \bar{m}_{k-1}(\xb)^2 - m^k(\xb)^2\\
= & \mathbb{E}_{I^k}\left[ \displaystyle\sum_{(i,j)\in [\![ p ]\!] \times [\![ p ]\!]} \displaystyle\frac{d^k_i d^k_j}{\pi^k_i \pi_j^k} \ind{(i,j)\in I^k}\right] + 2\bar{m}_{k-1}(\xb)\left[m^k(\xb) - \bar{m}_{k-1}(\xb)\right] \\
& + \bar{m}_{k-1}(\xb)^2 - m^k(\xb)^2\\
= &\mathbb{E}_{I^k}\left[ \displaystyle\sum_{(i,j)\in [\![ p ]\!] \times [\![ p ]\!]} \displaystyle\frac{d^k_i d^k_j}{\pi^k_i \pi_j^k} \ind{(i,j)\in I^k}\right] - (\bar{m}_{k-1}(\xb)-m^k(\xb))^2
 \\
 = & \displaystyle\sum_{(i,j)\in [\![ p ]\!]
 \times [\![ p ]\!]} \displaystyle\frac{\pi_{ij}^k}{\pi^k_i\pi_j^k} d^k_id^k_j - \displaystyle\sum_{(i,j)\in [\![ p ]\!] \times [\![ p ]\!]} d^k_id^k_j\\ 
 = & \displaystyle\sum_{(i,j)\in [\![ p ]\!] \times [\![ p ]\!]} \left(\displaystyle\frac{\pi_{ij}^k}{\pi^k_i\pi_j^k}-1\right) d^k_id^k_j .
\end{array} 
$$
\end{proof}

With this expression for the variance of the ameliorated model at a point $\xb$, 
a reasonable goal is to minimize the variance in \cref{eq:var} as a function of the probabilities $\{\pi^k_i\}_{i\in [p]}$ and $\{\pi_{i,j}^k\}_{i,j\in [p]\times [p]}$.
This choice, of course, leads to two immediately transparent issues:
\begin{enumerate}
 \item The variance is expressed pointwise and depends on differences $d_i^k$ between two model predictions at a single point $\xb$. 
 We should be interested in a more global quantity.
 In particular, for each component function $F_i$, when constructing the ameliorated model $\hat{m}_{I^k}$, we should be concerned with the value of
 \begin{equation}\label{eq:global1}
 d^{k,I^k}_{i} = \displaystyle\max_{\xb\in\cB(\xb^k;\Delta_k)} \left|m_i(\xb;\xb^k) - m_i(\xb;\cb_i^{k-1})\right|,
 \end{equation}
 whereas when constructing $\hat{m}_{J^k}$, we should be particularly concerned with the value of
 \begin{equation}\label{eq:global2}
 d^{k,J^k}_{i} = \max\left\{\left|m_i(0;\xb^k) - m_i(0;\cb_i^{k-1})\right|, \, \left|m_i(\sba^k;\xb^k) - m_i(\sba^k;\cb_i^{k-1})\right|\right\}.
 \end{equation}
 \item We cannot evaluate the difference $d_i^k$ (and hence the quantities in \cref{eq:global1} or \cref{eq:global2}) without first constructing $m_i(\xb;\xb^k)$---which, in the case of \cref{eq:first_order}, requires an evaluation of $F_i(\xb^k)$ and $\nabla F_i(\xb^k)$. 
 These additional evaluations undermine our motivation for sampling component functions in the first place.
\end{enumerate}

In the remainder of this paper we address the first of these two issues by replacing $d^k_i$ in \cref{eq:var} with one of the two quantities in \cref{eq:global1} or \cref{eq:global2} when constructing the respective estimator $\hat{m}_{I^k}(\xb)$ of $\hat{m}_{J^k}(\xb)$. 
For simplicity of notation, we will continue to write $d^k_i$ in our variance expressions, but the interpretation should be whichever of these two local error bounds is appropriate. 

To handle the second issue, we propose computing an upper bound on $d_i^{k,I^k}$ or $d_j^{k,J^k}$. 
In general, we observe that
\begin{align}\label{eq:upper_bounding_dki}
d^k_i =  m_i(\xb;\xb^k)-m_i(\xb;\cb_i^{k-1}) & \leq |m_i(\xb;\xb^k)-m_i(\xb;\cb_i^{k-1})|\\ \nonumber
& \leq |f_i(\xb) - m_i(\xb;\xb^k)| + |f_i(\xb) - m_i(\xb;\cb_i^{k-1})|\\ \nonumber
& =: e(\xb;\xb^k,\cb_i^{k-1}).\\ \nonumber
\end{align}
Continuing with our motivating example of \cref{eq:first_order},
and under \Cref{ass:f}, we may upper bound 
$|F_i(\xb) - m_i(\xb;\cb)| \leq \frac{L_i}{2}\|\xb-\cb\|^2$. 
Thus, in the \cref{eq:first_order} case, the bound in \cref{eq:upper_bounding_dki} may be continued as
\begin{equation*}
d_i^k \leq e(\xb;\xb^k,\cb_i^{k-1}) \leq \frac{L_i}{2}\left(\|\xb-\xb^k\|^2 + \|\xb-\cb_i^{k-1}\|^2 \right).
\end{equation*}
Moreover, we may then upper bound \cref{eq:global1} as
\begin{equation}\label{eq:global1_fo}
 d^{k,I^k}_i = \max_{\xb\in\cB(\xb^k;\Delta_k)} e(\xb;\xb^k;\cb_i^{k-1}) \leq \frac{L_i}{2}\left(\Delta_k^2 + (\|\xb^k-\cb^{k-1}_i\| + \Delta_k)^2\right),
\end{equation} 
and we may upper bound \cref{eq:global2} as
\begin{equation}\label{eq:global2_fo}
 d^{k,J^k}_i \leq \frac{L_i}{2}\max\{\|\xb^k-\cb_i^{k-1}\|^2, \|\sba^k-\xb^k\|^2 + \|\xb^k + \sba^k - \cb^{k-1}_i\|\}.
\end{equation}

Assuming $L_i$ is known, we have now resolved both issues, by having computable upper bounds in \cref{eq:global1_fo} and \cref{eq:global2_fo}.
Now, when we replace $d^k_i$ in \cref{eq:var} with either $d^{k,I^k}_i$ or $d^{k,J^k}_i$ and subsequently minimize the expression with respect to the probabilities $\{\pi^k_i\}, \{\pi_{ij}^k\}$, we are minimizing an upper bound on the variance over a set (the set being $\cB(\xb^k;\Delta_k)$ and $\{\xb^k,\xb^k+\sba^k\}$ when working with $d^{k,I^k}_i$ and $d^{k,J^k}_i$, respectively). 

We note that, especially in settings of derivative-free optimization, assuming access
to $L_i$ is often impractical. 
Although we will motivate
a particular choice of probabilities $\{\pi^k_i\}$ assuming access to $L_i$, we will demonstrate in \Cref{sec:nolip} 
that a simple scheme for dynamically estimating $L_i$ -- and in turn approximating the particular choice of $\{\pi^k_i\}$ -- suffices in practice. 

\subsection{A proposed method for choosing probabilities $\pi_i^k$ given a fixed batch size}\label{sec:optprob}
In Proposition~\ref{prop:unbiased}, we established that any set of nonzero inclusion probabilities $\{\pi_i^k\}_{i=1}^p$ employed in the construction of $\hat{m}_{I_k}$ and $\hat{m}_{J_k}$ will yield an unbiased estimator. 
Some unbiased models, however, will naturally be better than others. 
As is standard in statistics, a model of least (or, at least, low) variance -- an expression for which was computed in Proposition~\ref{prop:var} -- is certainly preferable. 
Because we are considering a setting where there is likely a budget on computational resource use per iteration of an optimization method - as constrained by, for instance, the availability of parallel resources -- we arrive at a high-level goal of seeking a low-variance unbiased estimator of the model subject to a constraint on the number $(|I_k|), (|J_k|)$ of component model updates we are able to perform. 
Towards achievement of this goal, we propose a particular means of determining $\pi_i^k$ in this section, but remind the reader again that any nonzero probabilities will satisfy the minimum requirements for unbiasedness. 

We begin by considering the setting of \emph{independent sampling} of batches, previously discussed in an optimization setting in \cite{CsibaRichtarik2018,HanzelyRichtarik2019,HorvathRichtarik2019}.
This setting is also sometimes referred to as Poisson sampling in the statistics literature. 
In independent sampling, there is an independent Bernoulli trial with success probability $\pi^k_i$ associated with each of the $i$ component functions;
a single realization of the $p$ independent Bernoulli trials determines which of the $p$ component functions are included in $I^k$. 
Thus, under independent sampling, $\pi_{ii}^k = \pi^k_i$ for all $i$, and $\pi_{ij}^k = \pi^k_i\pi_j^k$ for all $(i,j)$ such that $i\neq j$. 
Notably, under the assumption of independent sampling, the variance of $\hat{m}_{I^k}(\xb)$ in \cref{eq:var} established in Proposition~\ref{prop:var} greatly simplifies to
\begin{equation}
\label{eq:is_var}
\displaystyle\Vbb_{I^k}\left[\hat{m}_{I^k}(\xb)\right] = \displaystyle\sum_{i=1}^p \left(\displaystyle\frac{1}{\pi^k_i} -1\right)(d^k_i)^2\quad \forall x \in\Reals^n.
\end{equation}
As a sanity check, notice that for any set of nonzero probabilities $\pi^k_i\in(0,1]$, the variance in \cref{eq:is_var} is nonnegative. 
As a second sanity check, notice that if we deterministically update every model on every iteration $k$ (that is, $I_k = \{1,\dots,p\}$ for all $k$), then $\pi^k_i=1$ for all $k$ and for all $i$, and the 
variance of the estimator $\hat{m}_{I^k}(\xb)$ is 0 for all $k$. 
As an immediate consequence of the independence of the Bernoulli trials,
\begin{equation}\label{eq:expected_batchsize} 
\Ebb_{I^k}\left[\left|I^k\right|\right] = \displaystyle\sum_{i=1}^p \pi^k_i.
\end{equation}

Assuming independent sampling, and in light of \cref{eq:expected_batchsize},  
we specify a \emph{batch size parameter} $b$ and constrain
$\displaystyle\sum_i \pi^k_i=b.$ 
An estimator $\hat{m}_{I^k}(\xb)$ of least variance with \emph{expected} batch size $|I^k|=b$ is one defined by $\{\pi^k_i\}$ solving
\begin{equation}
\label{eq:opt_prob_p}
\begin{array}{rl}
\displaystyle\min_{\{\pi^k_i\}} & \displaystyle\sum_{i=1}^p \left(\displaystyle\frac{1}{\pi^k_i} -1\right)(d^k_i)^2\\
\text{s. to} & \displaystyle\sum_{i=1}^p \pi^k_i = b\\
& 0\leq \pi^k_i \leq 1 \quad \forall i=1,\dots,p.\\
\end{array}
\end{equation}
By deriving the Karush--Kuhn--Tucker (KKT)   conditions, one can see that the solution to \cref{eq:opt_prob_p} is defined, for each $i$, by
\begin{equation}
\label{eq:opt_probs}
\pi_{(i)}^k = \left\{ 
\begin{array}{rl}
(b + c - p) \displaystyle\frac{|d^k_{(i)}|}{\displaystyle\sum_{j=1}^c |d^k_{(j)}|} &\text{ if } i\leq c\\
1 & \text{ if } i > c,\\
\end{array}\right.
,
\end{equation}
where $c$ is the largest integer satisfying 
\begin{equation}
\label{eq:kkt_ineqs}
0<b+c-p\leq \displaystyle\sum_{i=1}^c \displaystyle\frac{|d^k_{(i)}|}{|d^k_{(c)}|},
\end{equation}
 and we have used the order statistics notation
$|d^k_{(1)}| \leq |d^k_{(2)}|\leq \dots \leq |d^k_{(p)}|.$ 
Observe that $c$ is well-defined; 
in the worst case, $c=p-b+1$ satisfies the strict inequality in \cref{eq:kkt_ineqs}, and in turn, $b + (p-b+1) -p = 1$, which is a trivial lower bound on the right-hand side of \cref{eq:kkt_ineqs}, since the $c$th term in the positive sum is 1.

However, independent sampling is undesirable
because it means we can  exert control only over
the \emph{expected} size of $I^k$ over all draws of $I^k$. 
We briefly record a known result concerning a Chernoff bound for sums of independent Bernoulli trials; see, e.g., \cite{vershynin2018high}[Theorem 2.2.2].
 \begin{proposition}\label{prop:chernoff}
For all $\delta >0$, 
$$\mathbb{P}_{I^k} \left[\left|I^k\right| \geq (1+\delta)b \right] \leq \exp\left(-\displaystyle\frac{b\delta^2}{2+\delta}\right).$$
For all $\delta\in(0,1)$,
$$ \mathbb{P}_{I^k} \left[\left|I^k\right| \leq (1-\delta)b \right] \leq \exp\left(-\displaystyle\frac{b\delta^2}{2} \right).$$
\end{proposition} 
When $b=16$, for instance, we see from \Cref{prop:chernoff} that the probability of obtaining a batch twice as large (that is, satisfying $|I^k|\geq 32$) is slightly less than $0.0025$. 
In practical situations, however, this can be problematic.
If one has $b$ parallel resources available for computation, then one does not want to underutilize---or, worse, attempt to overutilize---the resources by having a realization of $I^k$ be too small or too large, respectively. 


Therefore,  when we must compute a batch of a fixed size $b$ due to computational constraints, 
we consider a process called \emph{conditional Poisson sampling}, which is described in \Cref{alg:poisson}. 

\begin{algorithm}[h!]
\caption{Conditional Poisson sampling for selecting $I_k$ or $J_k$} 
\label{alg:poisson}
\textbf{Input: } Batch size $b>0$, error bounds $\{d_i^k\}_{i=1}^p.$\\
Sort $d_{(1)}^k\leq \dots \leq d_{(p)}^k$.\\
Compute probabilities $\{\pi^k_i\}_{i=1}^p$ according to \cref{eq:opt_probs}.\\
Transform the probabilities $\{\pi^k_i\}$ into $\{\tilde{\pi}^k_i\}$ using \Cref{alg:deville}.\\
$A\gets\emptyset$.\\
\While{$|A|\neq b$}{
Perform independent sampling with transformed probabilities $\tilde{\pi}^k_i$ to obtain $A$.
}
\textbf{Output: }$A$, original probabilities $\{\pi^k_i\}$.
\end{algorithm}

\Cref{alg:deville} is stated in \Cref{sec:deville}.
\Cref{alg:deville} was developed over several papers \cite{chen1994weighted,chen2000general}
and provably takes a set of desired inclusion probabilities $\pi_i^k$ and transforms them such that the randomized output of \Cref{alg:poisson}, $I_k, J_k$, satisfies the desired inclusion probabilities. 

We make two remarks.
Firstly, \Cref{alg:poisson} is a rejection method, which might raise concerns about stopping time. 
However, in the details of \Cref{alg:deville}, there is a degree of freedom in the transformation that allows us to normalize $\{\tilde{\pi}^k_i\}$ such that they sum to $b$. 
Thus, the expected size of $A$, as obtained by independent sampling, is $b$ in each iteration of the while loop. 
Therefore, it is reasonable that this rejection sampling has a short expected stopping time; this is observed in practice, and is so unconcerning that we don't empirically demonstrate this. 

Secondly, although we are guaranteed (up to finite precision) that $\mathbb{P}[i\in I_k] = \pi^k_i$ for $I_k$ sampled by \Cref{alg:poisson}, there is certainly no guarantee that $\mathbb{P}[i,j\in I_k] = \pi^k_{ij}$.
However, as noted in \cite{aires1999algorithms}, these second-order inclusion probabilities are often remarkably close to $\pi^k_{ij}$.
In \Cref{sec:deville}, we provide an inexpensive formula derived from these works for computing the second-order inclusion probabilities associated with conditional Poisson samples generated by \Cref{alg:poisson}.

\subsection{Additional models beyond \cref{eq:first_order}}
\label{eq:beyondfo}

The results derived in \Cref{sec:optprob} are valid for other classes of models beyond \cref{eq:first_order}. 
For any new class of models $m_i(\xb;\cb^k_i)$ one can apply the same development as used for \cref{eq:first_order} provided one can derive a (meaningful) upper bound on $d_i^{k,I^k}$ (\cref{eq:global1}) and $d_i^{k,J^k}$ (\cref{eq:global2}), such as those  in \cref{eq:global1_fo} and \cref{eq:global2_fo} for \cref{eq:first_order}.
In the following three subsections we introduce additional classes of models and derive the corresponding upper bounds on variance. 

\subsubsection{Linear interpolation models}
For a second class of models, motivated by model-based derivative-free (also known as zeroth-order) optimization, we consider models of the form
\begin{equation}
\label{eq:zeroth_order}
m_i(\xb;\cb^k_i) = F_i(\cb^k_i) + \gb_i^k(\cb^k_i;\delta_i)^\top (\xb - \cb^k_i), \tag{ZO}
\end{equation}
where $\gb_i^k(\cb^k_i;\delta_i)$ denotes an approximate gradient computed by linear interpolation on a set of interpolation points contained in a ball $\cB(\cb^k_i;\delta_{i})$. 
The use of the notation $\delta_i$, as opposed to $\Delta_k$, is intended to denote that $\delta_i$ is a parameter that is potentially updated independently of the iteration $k$ of \Cref{alg:dfotr}. 

On iterations where $i\in I^k$, one would update a model $m_i(\xb;\cb^k_i)$ of the form \cref{eq:zeroth_order} to be fully linear on $\cB(\xb^k;\Delta_k)$. Algorithms for performing such updates are common in model-based derivative-free optimization literature; see, for instance, \cite[Chapter 3]{Conn2009a}. 
Suppose the model gradient term $\gb(\cb_i^k;\delta_i)$ of  \cref{eq:zeroth_order} is constructed by linear interpolation on a set of $n+1$ points
$Y_i = \{\cb_i^k=\vb^0, \vb^1,\dots,\vb^n\} \subset \cB(\cb_i^k;\delta_i)$, where $\delta_i$ denotes the value of $\Delta_j$ on the last iteration $j\leq k$, where $i\in I_j$. It will be convenient to assume the following about sets of interpolation points.

\begin{assumption}
\label{ass:Y}
The set of points $\{\cb_i^k,\vb^1,\dots,\vb^n\} \subset \cB(\cb^k_i;\delta_i)$ is poised for linear interpolation.
\end{assumption}

We denote, for any $\xb$ of interest, 
\begin{equation}\label{eq:vmatrix}
 V_{Y_i} = \left[\vb^1 - \cb_i^k, \;\; \dots, \;\; \vb^n-\cb_i^k\right],
\end{equation}
and
\begin{equation}\label{eq:vhatbound}
\hat{V}_{Y_i}(\xb;\cb^k_i) = \frac{1}{\max\{\delta_i,\|\xb-\cb^k_i\|\}}V_{Y_i}.
\end{equation} 
We remark that \cref{eq:vmatrix} has no dependence on $\xb$, and hence the notation $V_{Y_i}$ does not involve $\xb$.
This is in contrast to \cref{eq:vhatbound}, where $\xb$ is explicitly involved in the scaling factor. 
Under \Cref{ass:Y}, both of the matrices \cref{eq:vmatrix} and \cref{eq:vhatbound} are invertible. 
In this setting we obtain the following by altering the proofs of \cite[Theorems 2.11, 2.12]{Conn2009a}.

\begin{theorem}
\label{thm:dfo_bound}
Under \Cref{ass:f} and \Cref{ass:Y}, it follows that for all $\xb\in \Reals^n$, 
\begin{equation}
\label{eq:dfo_bound1}
F_i(\xb) - m_i(\xb;\cb_i^{k})
\leq 
 \frac{3L_i}{2}\|\xb-\cb^k_i\|^2 + \frac{L_i\sqrt{n}\|\hat{V}_{Y_i}^{-1}(\xb;\cb^k_i)\|}{2\max\{\delta_i,\|\xb-\cb^k_i\|\}}\delta_i^2\|\xb-\cb^k_i\| .
\end{equation}
\end{theorem}

The proof is left to \Cref{sec:thm1}, since it is not particularly instructive on its own. 

On iterations such that $i\in I^k$, 
the set of points $Y_i$ would also be updated, first to include $\xb^k$ and then to guarantee poisedness of the updated $Y_i$ on $\cB(\xb^k;\Delta_k)$. This poisedness can be achieved, for instance, by using the same sampling and model improvement techniques as those employed in \texttt{POUNDERS}.
Thus, we may similarly conclude from the proof of \Cref{thm:dfo_bound} that 
\begin{equation}\label{eq:dfo_bound2}
\begin{array}{lll}
F_i(\xb) - m_i(\xb; \xb^k) & \leq &
\frac{3L_i}{2}\|\xb-\xb^k\|^2 + \frac{L_i\sqrt{n}\|\hat{V}_{Y_i}^{-1}(\xb;\xb^k)\|}{2\max\{\Delta_k,\|\xb-\xb^k\|\}}\Delta_k^2\|\xb-\xb^k\|\\ .
\end{array}
\end{equation}
Combining \cref{eq:dfo_bound1} and \cref{eq:dfo_bound2},  we have 
\begin{equation*}
\begin{array}{ll}
e(\xb;\xb^k,\cb_i^{k-1}) & \leq  \frac{3L_i}{2}\|\xb-\cb^k_i\|^2 + \frac{L_i\sqrt{n}\|\hat{V}_{Y_i}^{-1}(\xb;\cb^k_i)\|}{2\max\{\delta_i,\|\xb-\cb^k_i\|\}}\delta_i^2\|\xb-\cb^k\|\\
& + \frac{3L_i}{2}\|\xb-\xb^k\|^2 + \frac{L_i\sqrt{n}\|\hat{V}_{Y_i}^{-1}(\xb;\xb^k)\|}{2\max\{\Delta_k,\|\xb-\xb^k\|\}}\Delta_k^2\|\xb-\xb^k\|.
\end{array}
\end{equation*}
Recalling that $\|\hat{V}^{-1}_{Y_i}(\xb;\cb^k_i)\|/\max\{\delta_i,\|\xb-\cb^k_i\|\} = \|V^{-1}_{Y_i}\|$,
\begin{equation*}
\begin{array}{lll}
d^{k,I^k}_i & = & \displaystyle\max_{\xb\in\cB(x^k;\Delta_k)} e(\xb;\xb^k;\cb_i^{k-1}) \\
& \leq & L_i\left(\displaystyle\frac{3}{2}(\|\xb^k-\cb^k_i\| + \Delta_k)^2 + \frac{\sqrt{n}\|V_{Y_i}^{-1}\|}{2}\delta_i^2(\|\xb^k-\cb^k_i\|+\Delta_k)\right)\\
&& + L_i\left(\frac{3}{2}\Delta_k^2 + \frac{\sqrt{n}\|V^{-1}_{Y_i}\| }{2}\Delta_k^3\right) \\
\end{array}
\end{equation*}
and
\begin{equation*}
\begin{array}{lll}
d^{k,J^k}_i & =  &\max\left\{e(\xb^k;\xb^k;\cb_i^{k-1}), e(\xb^k+\sba^k;\xb^k;\cb_i^{k-1})\right\} \\
& \leq & L_i\max\left\{\frac{3}{2}\|\xb^k-\cb^k_i\|^2 + \frac{\sqrt{n}\|V_{Y_i}^{-1}\|}{2}\delta_i^2\|\xb^k-\cb^k_i\|\right.,\\
& & \frac{3}{2}\|\xb^k+\sba^k-\cb^k_i\|^2 + \frac{\sqrt{n}\|V_{Y_i}^{-1}\|}{2}\delta_i^2\|\xb^k+\sba^k-\cb^k_i\|\\
&  & + \left. \frac{3}{2}\|\sba^k\|^2 + \frac{\sqrt{n}\|V_{Y_i}^{-1}\|}{2}\Delta_k^2\|\sba^k\| \right\}.
\end{array}
\end{equation*}

\subsubsection{Gauss--Newton models \cref{eq:gauss_newton}} 
For a third class of models, we consider the case where each $F_i(\xb)$ in \cref{eq:sum} is of the form $\frac{1}{2}f_i(\xb)^2$. 
In this (nonlinear) least-squares setting, we can form the Gauss--Newton model by letting the $i$th model be defined as 
\begin{equation}
\label{eq:gauss_newton}
\begin{array}{rl}
m_i(\xb;\cb^k_i) =& \frac{1}{2}\left(f_i(\cb^k_i) + \nabla f_i(\cb^k_i)^\top (\xb-\cb^k_i)\right)^2 \\
 =& \frac{1}{2}f_i(\cb^k_i)^2  + f_i(\cb^k_i)\nabla f_i(\cb^k_i)^\top(\xb-\cb^k_i) \\
& + \frac{1}{2}(\xb-\cb^k_i)^\top\left(\nabla f_i(\cb^k_i)\nabla f_i(\cb^k_i)^\top\right)\left(\xb-\cb^k_i\right).
\end{array}
 \tag{FOGN} 
\end{equation} 
As in the case of \cref{eq:first_order}, updating a model of the form \cref{eq:gauss_newton} entails a function and gradient evaluation of $f_i$ at $\xb^k$. 

Noting that the second-order Taylor model of $F_i(\xb)$ centered on $\xb^k$ is given by
\begin{equation*}
\begin{array}{l}
\frac{1}{2}f_i(\xb^k)^2 + f_i(\xb^k)\nabla f_i(\xb^k)^\top (\xb-\xb^k) \\
+ \frac{1}{2}(\xb-\xb^k)^\top \left[ \nabla f_i(\xb^k)\nabla f_i(\xb^k)^\top +  f_i(\xb^k) \nabla^2 f_i(\xb^k)\right](\xb-\xb^k),
\end{array}
\end{equation*}
we can derive that
$$|F_i(\xb) - m_i(\xb;\xb^k)| \leq \displaystyle\frac{L_{\nabla^2 f_i}}{6}\|\xb-\xb^k\|^3 + \frac{1}{2}|f_i(\xb^k)|L_{\nabla f_i}\|\xb-\xb^k\|^2,$$
where $L_{\nabla^2 f_i}$ and $L_{\nabla f_i}$ denote (local) Lipschitz constants of $\nabla^2 f_i$ and $\nabla f_i$, respectively.
To avoid having to estimate $L_{\nabla^2 f_i}$, and justified in part because $\|\xb-\xb^k\|^3$ is dominated by $\|\xb-\xb^k\|^2$ in the limiting behavior of \Cref{alg:dfotr},\footnote{As per \cite[Theorem 4.11]{Chen2017}, as $k\to\infty$ in \Cref{alg:dfotr}, $\Delta_k\to 0$ almost surely.} 
we make the (generally incorrect) simplifying assumption that $L_{\nabla^2 f_i}=0$. Under this simplifying assumption, 
$$
e(\xb;\xb^k,\cb_i^{k-1}) \leq
\displaystyle\frac{|f_i(\cb^k_i)|L_{\nabla f_i}}{2}\|\xb-\cb^k_i\|^2 + 
\displaystyle\frac{|f_i(\xb^k)|L_{\nabla f_i}}{2}\|\xb-\xb^k\|^2.
$$
Because we do not know $|f_i(\xb^k)|$, we upper bound it by noting that
$$
\begin{array}{lll}
|f_i(\xb^k)| &\leq & \left|f_i(\cb^k_i) + \nabla f_i(\cb_i^k)^\top (\cb^k_i-\xb^k) + \frac{L_{\nabla f_i}}{2}\|\cb^k_i-\xb^k\|^2\right| \\
& \leq & \left|f_i(\cb_i^k)\right| + \left|\nabla f_i(\cb_i^k)^\top(\cb^k_i - \xb^k)\right| + \frac{L_{\nabla f_i}}{2}\left\|\cb^k_i-\xb^k\right\|^2 =: M\left(\xb^k,\cb^k_i\right),
\end{array}
$$
to arrive at
 \begin{equation*}
 e(\xb;\xb^k,\cb_i^{k-1}) \leq 
 \displaystyle\frac{|f_i(\cb^k_i)|L_{\nabla f_i}}{2}\|\xb-\cb^k_i\|^2 + 
\displaystyle\frac{L_{\nabla f_i}M(\xb^k,\cb^k_i)}{2}\|\xb-\xb^k\|^2.
 \end{equation*}
Thus,
\begin{equation*}
d_i^{k,I^k} \leq  \displaystyle\frac{|f_i(\cb^k_i)|L_{\nabla f_i}}{2}\left(\|\xb^k-\cb^k_i\| + \Delta_k\right)^2 + 
\displaystyle\frac{L_{\nabla f_i}M(\xb^k,\cb^k_i)}{2}\Delta_k^2.
\end{equation*}
and
\begin{equation*}
d_i^{k,J^k} \leq \displaystyle\frac{L_{\nabla f_i}}{2}\max\left\{ |f_i(\cb^k_i)|\|\xb^k-\cb^k_i\|^2,  
 |f_i(\cb^k_i)|\|\xb^k + \sba^k-\cb^k_i\|^2 + 
M(\xb^k,\cb^k_i)\|\sba^k\|^2\right\}.
\end{equation*}

\subsubsection{Zeroth-order Gauss--Newton models \cref{eq:zero_gauss_newton}} 
Here, as in \texttt{POUNDERS}, we consider the case where each $F_i(\xb)$ in \cref{eq:sum} is of the form $\frac{1}{2}f_i(\xb)^2$.
Rather than having access to first-order information, however, we construct a zeroth-order model as in \cref{eq:zeroth_order}, leading to a zeroth-order Gauss--Newton model
\begin{equation}
\label{eq:zero_gauss_newton}
\begin{array}{rl}
m_i(\xb;\cb_i^k) =& \frac{1}{2}\left(f_i(\cb_i^k) + \gb_i^k(\cb_i^k;\delta_i)^\top(\xb-\cb^k_i) \right)^2\\ 
 =& \frac{1}{2}f_i(\cb^k_i)^2  + f_i(\cb^k_i) \gb_i^k(\cb_i^k;\delta_i)^\top(\xb-\cb^k_i) \\
& + \frac{1}{2}(\xb-\cb^k_i)^\top\left( \gb_i^k(\cb_i^k;\delta_i) \gb_i^k(\cb_i^k;\delta_i)^\top\right)(\xb-\cb^k_i), 
\end{array}
\tag{ZOGN}
\end{equation}
where $\gb_i^k(\cb_i^k;\delta_i)$ is constructed such that $f_i(\cb_i^k)+\gb_i^k(\cb_i^k;\delta_i)^\top(\xb-\cb^k_i)$ is a fully linear model of $f_i$ on $\cB(\cb^k_i;\delta_i)$, for example, by linear interpolation.
The choice of notation $\delta_i$ is meant to suggest that the same linear interpolation as used in \cref{eq:zeroth_order} is applicable to obtain a model of $f_i$. 
\texttt{POUNDERS} in fact employs an additional fitted quadratic term $H_i^k(\cb_i^k;\delta_i)$ in \cref{eq:zero_gauss_newton}; for simplicity in presentation, we assume that no such term is used here. 
However, as long as one supposes that $H_i^k(\cb_i^k;\delta_i)$ is uniformly (over $k$) bounded in spectral norm, then analogous results are easily derived. 

We note that given a bound on $\|\nabla f_i(\cb^k_i) - \gb(\cb^k_i;\delta_i)\|$, we immediately attain a bound
$$\|\nabla F_i(\cb^k_i) - f_i(\cb^k_i)\gb(\cb^k_i;\delta_i)\| = |f_i(\cb^k_i)|\|\nabla f_i(\cb^k_i) - \gb(\cb^k_i;\delta_i)\|. $$ 
Thus, using the same notation as in \cref{eq:zeroth_order} and \cref{eq:gauss_newton}, we can follow the proof of \Cref{thm:dfo_bound} starting from \cref{eq:grad_at_center_bound} to conclude
\begin{equation*}
F_i(\xb) - m_i(\xb;\cb_i^{k})
\leq 
 \frac{3L_i|f_i(\cb^k_i)|}{2}\|\xb-\cb^k_i\|^2 + \frac{L_i|f_i(\cb^k_i)|\sqrt{n}\|\hat{V}_{Y_i}^{-1}(\xb;\cb^k_i)\|}{2\max\{\delta_i,\|\xb-\cb^k_i\|\}}\delta_i^2\|\xb-\cb^k_i\|.
\end{equation*}

We can similarly conclude from the logic employed for \cref{eq:zeroth_order} that
\begin{equation}\label{eq:global1_zogn}
\begin{array}{lll}
d^{k,I^k}_i 
& \leq & L_i|f_i(\cb^k_i)|\left(\displaystyle\frac{3}{2}(\|\xb^k-\cb^k_i\| + \Delta_k)^2 + \frac{\sqrt{n}\|V_{Y_i}^{-1}\|}{2}\delta_i^2(\|\xb^k-\cb^k_i\|+\Delta_k)\right)\\
&& + L_i|f_i(\cb^k_i)|\left(\frac{3}{2}\Delta_k^2 + \frac{\sqrt{n}\|V^{-1}_{Y_i}\| }{2}\Delta_k^3\right) \\
\end{array}
\end{equation}
and
\begin{equation}\label{eq:global2_zogn}
\begin{array}{lll}
d^{k,J^k}_i & \leq & L_i|f_i(\cb^k_i)|\max\left\{\frac{3}{2}\|\xb^k-\cb^k_i\|^2 + \frac{\sqrt{n}\|V_{Y_i}^{-1}\|}{2}\delta_i^2\|\xb^k-\cb^k_i\|\right.,\\
& & \frac{3}{2}\|\xb^k+\sba^k-\cb^k_i\|^2 + \frac{\sqrt{n}\|V_{Y_i}^{-1}\|}{2}\delta_i^2\|\xb^k+\sba^k-\cb^k_i\|\\
&  & + \left. \frac{3}{2}\|\sba^k\|^2 + \frac{\sqrt{n}\|V_{Y_i}^{-1}\|}{2}\Delta_k^2\|\sba^k\| \right\}.
\end{array}
\end{equation}

\subsection{Convergence guarantees}\label{sec:storm} 
For brevity, we do not provide a full proof of convergence of \Cref{alg:dfotr} but instead appeal to the first-order convergence results for \texttt{STORM} \cite{Chen2017, BCMS2018}.
The \texttt{STORM} framework is a trust-region method with stochastic models and function value estimates;
 \Cref{alg:dfotr} can be seen as a special case of \texttt{STORM}, 
where the stochastic models are given by the ameliorated model $\hat{m}_{I_k}$ and the stochastic estimates are computed via $\hat{m}_{J_k}$. 

With these specific choices of models and estimates, we record the following definitions \cite[Section 3.1]{BCMS2018}. 

\begin{definition}
The ameliorated model $\hat{m}_{I_k}$ is \emph{$(\kappa_f,\kappa_g)$-fully linear of $f$ on $\cB(\xb^k;\Delta_k)$} provided that for all $\yb\in\cB(\xb^k;\Delta_k)$,
\begin{equation*}
|f(\yb)-\hat{m}_{I_k}(\yb)| \leq \kappa_f\Delta_k^2 \quad \text{ and } \quad \|\nabla f(\yb) - \nabla \hat{m}_{I_k}(\xb^k)\| \leq \kappa_g \Delta_k.
\end{equation*}
\end{definition}

\begin{definition}
The ameliorated model values $\{\hat{m}_{J_k}(\xb^k),\hat{m}_{J_k}(\xb^k+\sba^k)\}$ are \emph{$\epsilon_f$-accurate estimates of $\{f(\xb^k),f(\xb^k+\sba^k)\}$}, respectively, 
provided that given $\Delta_k$, both
\begin{equation*}
|\hat{m}_{J_k}(\xb^k) - f(\xb^k)| \leq \epsilon_f\Delta_k^2 \quad \text{ and } \quad |\hat{m}_{J_k}(\xb^k + \sba^k) - f(\xb^k + \sba^k)| \leq \epsilon_f\Delta_k^2.
\end{equation*}
\end{definition}

Denote by $\mathcal{F}_{k-1}$ the $\sigma$-algebra generated by the models 
$\{\hat{m}_{I_0}, \hat{m}_{I_1}, \dots, \hat{m}_{I_{k-1}}\}$
 and the estimates 
$\{\hat{m}_{J_0}(\xb^0), \hat{m}_{J_0}(\xb^0+\sba^0), \hat{m}_{J_1}(\xb^1), \hat{m}_{J_1}(\xb^1+\sba^1), \dots, \hat{m}_{J_{k-1}}(\xb^{k-1}), \hat{m}_{J_{k-1}}(\xb^{k-1}+\sba^{k-1})\}$.
Additionally denote by $\mathcal{F}_{k-1/2}$ the $\sigma$-algebra generated by the models $\{\hat{m}_{I_0}, \hat{m}_{I_1}, \dots, \hat{m}_{I_{k-1}},$
$\hat{m}_{I_k}\}$ and the estimates
$\{\hat{m}_{J_0}(\xb^0), \hat{m}_{J_0}(\xb^0+\sba^0), \hat{m}_{J_1}(\xb^1), \hat{m}_{J_1}(\xb^1+\sba^1), \dots, \hat{m}_{J_{k-1}}(\xb^{k-1}),$
$\hat{m}_{J_{k-1}}(\xb^{k-1}+\sba^{k-1})\}$.
That is, $\mathcal{F}_{k-1/2}$ is $\mathcal{F}_{k-1}$, but additionally including the model $\hat{m}_{I_k}$. 
Then, we may additionally define the following two properties of sequences. 

\begin{definition}
A sequence of random models $\{\hat{m}_{I_k}\}_{k=0}^\infty$ is \emph{$\alpha$-probabilistically $(\kappa_f,\kappa_g)$-fully linear with respect to the sequence $\{\cB(\xb^k;\Delta_k)\}_{k=0}^\infty$}
provided that for all $k$,
$$\mathbb{P}\left[\hat{m}_{I_k} \text{ is a } (\kappa_f,\kappa_g)\text{-fully linear model of } f \text{ on } \cB(\xb^k;\Delta_k) | \mathcal{F}_{k-1}  \right] \geq \alpha.$$
\end{definition}

\begin{definition}
A sequence of estimates $\{\hat{m}_{J_k}(\xb^k), \hat{m}_{J_k}(\xb^k+\sba^k)\}_{k=0}^\infty$ is \emph{$\beta$-probabilistically $\epsilon_f$ accurate with respect to the sequence $\{\xb^k,\Delta_k,\sba^k\}_{k=0}^\infty$} provided that for all $k$,
$$\begin{array}{l}
\mathbb{P}\left[ \hat{m}_{J_k}(\xb^k), \hat{m}_{J_k}(\xb^k+\sba^k)\} \text{ are } \epsilon_f \text{-accurate estimates of } \right.\\
\left.f(\xb^k) \text{ and } f(\xb^k+\sba^k)  | \mathcal{F}_{k-1/2}  \right] \geq \beta.
\end{array}$$
\end{definition}

With these definitions we can state a version of \cite[Theorem 4]{BCMS2018}. 

\begin{theorem}\label{thm:convergence}
Let \Cref{ass:f} hold. 
Fix $\kappa_f, \kappa_g > 0$, and fix $\epsilon_f \in [0, \kappa_f]$. 
Suppose that, uniformly over $k$, $\|\nabla^2 \hat{m}_{I_k}(\xb^k)\|\leq \kappa_{h}$  for some $\kappa_h\geq 0$. 
Then there exist $\alpha, \beta \in(\frac{1}{2}, 1)$ (bounded away from 1) such that, 
provided that $\{\hat{m}_{I_k}\}$ is $\alpha$-probabilistically $(\kappa_f,\kappa_g)$-fully linear with respect to the sequence $\{\cB(\xb^k;\Delta_k)\}$ generated by \Cref{alg:dfotr} and provided that
$\{\hat{m}_{J_k}(\xb^k), \hat{m}_{J_k}(\xb^k+\sba^k)\}_{k=0}^\infty$ is $\beta$-probabilistically $\epsilon_f$-accurate with respect to the sequence $\{\xb^k,\Delta_k,\sba^k\}_{k=0}^\infty$ generated by \Cref{alg:dfotr}, then 
the sequence $\{\xb^k\}_{k=0}^\infty$ generated by \Cref{alg:dfotr} satisfies, with probability 1, 
$$\displaystyle\lim_{k\to\infty} \|\nabla f(\xb^k)\| \to 0.$$ 
\end{theorem}

\cite[Theorem 4]{BCMS2018} additionally proves a convergence rate associated with \Cref{thm:convergence}, essentially on the order of $1/\epsilon^2$ many iterations to attain $\|\nabla f(\xb^k)\|\leq\epsilon$. 

Having computed the pointwise variance of the ameliorated models $\hat{m}_{I^k}$ and $\hat{m}_{J^k}$ in \cref{eq:is_var}, we may appeal directly to Chebyshev's inequality to obtain, at each $\xb\in \Reals^n$ and for any $C>0$, 
\begin{equation}\label{eq:chebyshev}
\Pbb\left[\left|\hat{m}_{I_k}(\xb) - m^k(\xb)\right| \geq C\Delta_k^2\right] \leq \displaystyle\frac{\displaystyle\Vbb_{I^k}\left[\hat{m}_{I_k}(\xb)\right]}{C^2\Delta_k^4} \qquad \forall \xb \in \Reals^n.
\end{equation}
We stress that \eqref{eq:chebyshev} holds for any $\xb\in \Reals^n$. 
Therefore, we can localize the conservative bound in \cref{eq:chebyshev} to the specific trust region $\cB(\xb^k; \Delta_k)$.
We have previously demonstrated, for each of the classes of models that we have considered, that $m^k$ is a 
 $(\kappa^{'}_f,\kappa^{'}_g)$-fully-linear model of $f$ on $\cB(\xb^k;\Delta_k)$ for appropriate constants $\kappa^{'}_f,\kappa^{'}_g$.
In particular,
 \begin{equation}
 \label{eq:generic_fl}
     |m^k(\xb) - f(\xb)| \leq \kappa^{'}_f \Delta_k^2 \qquad \forall \xb\in\cB(\xb^k; \Delta_k). 
 \end{equation}

Thus, we obtain from \cref{eq:chebyshev} and \cref{eq:generic_fl} that
\begin{equation}\label{eq:prob_fl}
\Pbb\left[\left|\hat{m}_{I_k}(\xb) - f(\xb)\right| \geq (C+\kappa^{'}_f)\Delta_k^2\right] \leq \displaystyle\frac{\displaystyle\Vbb_{I^k}\left[\hat{m}_{I_k}(\xb)\right]}{C^2\Delta_k^4}
\qquad \forall \, \xb \in \cB(\xb^k;\Delta_k).
\end{equation}

We remark that choosing $C + \kappa^{'}_f \leq \min\{\kappa_f, \epsilon_f\}$ in  \cref{eq:prob_fl} is not entirely sufficient to demonstrate that the model and estimate sequences employed by \Cref{alg:dfotr} are  $\alpha$-probabilistically $(\kappa_f,\kappa_g)$-fully linear and $\beta$-probabilistically $\epsilon_f$-accurate, respectively.
We identify two avenues for completing a proof of convergence for \Cref{alg:dfotr}. 
First, one could use the properties of a given class of models in tandem with the regularity of the component functions $F_i$ to show, via a union bound over a suitable covering set of points within $\cB(\xb^k;\Delta_k)$, that probabilistic full-linearity holds. 
Second, one could consider analyzing a variant of STORM that only assumes a condition resembling \cref{eq:prob_fl} holds on each iteration, as opposed to probabilistic full-linearity. 
Both of these routes would involve significant analysis, which is beyond the scope of this paper.
Nonetheless, based on \cref{eq:prob_fl} nearly resembling a fully-linearity condition, we suggest the scheme in \Cref{alg:dynamic_batchsize} for choosing $I_k, J_k$ in each iteration, which uses \Cref{alg:poisson} as a subroutine. 
We note that \Cref{alg:dynamic_batchsize} must terminate eventually because if $b=p$, then, as observed before, $V=0$. 

\begin{algorithm}[h!]
\caption{Determining a dynamic batch size for $I_k$ or $J_k$}
\label{alg:dynamic_batchsize}
\textbf{Input: } Computational resource size $r$, trust-region radius $\Delta_k$, accuracy constant $C>0$, probability parameter $\pi\in(\frac{1}{2},1)$. \\
$b\gets r$.\\
\While{true}{
Compute error bounds $\{d_i^k\}$. \\
Compute $A$ and $\{\pi^k_i\}$ using \Cref{alg:poisson} with batch size $b$ and error bounds $\{d_i^k\}$.\\
Compute approximate upper bound on variance according to \cref{eq:is_var}; call it $V$. \\
\If{$V > (1-\pi)C^2\Delta_k^4$}{
\vspace{6pt}
	$b \gets \min\{b + r,p\}$
}
\Else{\textbf{Return: } $A, $ probabilities $\{\pi^k_i\}$}
}
\end{algorithm} 

\section{Numerical Experiments}

We implemented a version of \Cref{alg:dfotr} in MATLAB 
and focus on the models \cref{eq:first_order} and \cref{eq:zero_gauss_newton}. 
Code is available at \url{https://github.com/mmenickelly/sampounders/}. 

\subsection{Test problems}
We focus on three simply structured objective functions in order to better study the behavior of \Cref{alg:dfotr}.  

\subsubsection{Logistic loss function} 
We test this function to explore models of the form \cref{eq:first_order} within \Cref{alg:dfotr}. 
Given a dataset $\{(\ab_{x,i}, \ab_{y,i})\}_{i=1}^p$, where each $(\ab_{x,i}, \ab_{y,i})\in\mathbb{R}^n\times\{-1,1\}$, we seek a classifier parameterized by $\xb\in\mathbb{R}^n$ that minimizes 
$f(\xb)$, where each component function has the form
$$F_i(\xb) = -\frac{1}{p}\log\left(\frac{1}{1+\exp(-\ab_{y,i}\ab_{x,i}^\top \xb)}\right) + \frac{\lambda}{2p}\|\xb\|^2,$$,
where $\lambda>0$ is a constant regularizer added in this experiment only to promote strong convexity (and hence unique solutions). 
We note that this model assumes that the bias term of the linear model being fitted is 0.
We randomly generate data via a particular method, which we now explain, inspired by the numerical experiments in \cite{NeedellSrebroWard2014}.
We first generate an optimal solution $\xb^*$ from a normal distribution with identity covariance centered at 0. 
We then generate data vectors $\ab_{x,i}$ according to one of three different modes of data generation:
\begin{enumerate}
\item \textbf{Imbalanced}: For $i=1,\dots,p$, each entry of $\ab_{x,i}$ is generated from a normal distribution with mean 0 and variance 1. 
The vector $\ab_{x,p}$ is then multiplied by 100. 
\item \textbf{Progressive}: For $i=1,\dots,p$ each entry of $\ab_{x,i}$ is generated from a normal distribution with mean 0 and variance 1 and is then multiplied by $i$. 
\item \textbf{Balanced}: For $i=1,\dots,p$, each entry of $\ab_{x,i}$ is generated from a normal distribution with mean 0 and variance 1. 
\end{enumerate}
We then generate $p$ random values $\{r_i\}_{i=1}^p$ uniformly from the interval $[0,1]$ and generate random labels $\{\ab_{y,i}\}$ via
\begin{equation*}
\ab_{y,i} = \left\{
\begin{array}{ll}
1 & \text{ if } r_i < \displaystyle\frac{1}{1 + \exp(-\ab_{x,i}^\top \xb^*)}\\
-1 & \text{ otherwise. }\\
\end{array}
\right.
\end{equation*}
It is well-known (see, e.g., \cite{Schmidt2013}[Section 5.1]) that $L_i$ in this logistic loss setting can be globally bounded as $L_i \leq \frac{1}{p}\left(\frac{\|\ab_{x,i}\|^2}{4} + \lambda\right)$. 
In our experiments we set $n=p=256$ and let $\lambda=0.1$. 

\subsubsection{Generalized Rosenbrock functions}\label{sec:genrosenbr}

Given a set of weights $\{\alpha_i\}_{i=1}^p$, where $p$ is an even integer, we define $f(\xb)$ componentwise as
\begin{equation}
\label{eq:rosen1}
F_i(\xb) =
\left\{
\begin{array}{ll}
 (10\alpha_i(\xb^2_{i} - \xb_{i+1}))^2 & \text{ if $i$ is odd}\\
(\alpha_i(\xb_{i-1}-1))^2\ & \text{ if $i$ is even.}\\
\end{array}
\right.
\end{equation}
The function $f(\xb)$ defined via \cref{eq:rosen1} satisfies $f(\xb^*)=0$ at the unique point $\xb^*=\oneb_n$, where $\oneb_n$ is the $n$-dimensional vector of ones.
We use this function to test models of the form \cref{eq:zero_gauss_newton}.
We observe that 
(recall, in our notation, $F_i(\xb) = f_i^2(\xb)$)
\begin{equation*}
[\nabla f_i(\xb)]_j = 
\begin{cases}
20\alpha_i\xb_j & \text{ if $i$ is odd, and $i=j$}\\
-10\alpha_i & \text{ if $i$ is odd, and $j=i+1$}\\
\alpha_i & \text{ if $i$ is even, and $j = i-1$}\\
0 & \text{ otherwise. } 
\end{cases}
\end{equation*}
Therefore, the Lipschitz constants $L_{i}$ satisfy $L_i = 0$ for all $i$ even.
Assuming that $\{\xb^k\}$ remains bounded in $[-1,1]^n$, we can derive an upper bound $L_{i} = 20\alpha_i$ for all $i$ odd. 

Similarly to the logistic loss experiments, we generate $\alpha_i$ according to three modes of generation:
\begin{enumerate}
\item \textbf{Imbalanced}: For $i=1,\dots,p-2$, $\alpha_i = 1$. We then choose $\alpha_{p-1}$ and $\alpha_p$ as $p$. 
\item \textbf{Progressive}: For $i=1,\dots,p$, $\alpha_i  = i$. 
\item \textbf{Balanced}: For $i=1,\dots,p$, $\alpha_i = 1$.
\end{enumerate}

To generate random problem instances, we simply change the initial point to be generated uniformly at random from $[-1,1]^p$. 
Once again, we generate 30 such random instances and run each variant of the algorithm applied to a random instance with three different 
initial seeds, yielding a total of 90 problems. 
We let $n=p=16$. 

\subsubsection{Cube functions}
Given a set of weights $\{\alpha_i\}_{i=1}^p$, we define $f(\xb)$ componentwise by
\begin{equation}
\label{eq:cube}
F_i(\xb) = \left\{
\begin{array}{ll}
(\alpha_1(x_1-1))^2 & \text{ if } i=1\\
\left(
\alpha_i(x_i-x_{i-1}^3)\right)^2 & \text{ if } i\geq 2.\\
\end{array}
\right. 
\end{equation}
The function $f(\xb)$ defined via \cref{eq:cube} satisfies $f(\xb^*)=0$ at the unique point $\xb^*=\oneb_n$.
We also use this function to test models of the form \cref{eq:zero_gauss_newton}. 
The gradient of $\nabla f_i$ is given coordinatewise as
\begin{equation*}
\left[ \nabla f_i(\xb)\right]_j = \left\{
\begin{array}{ll}
\alpha_i & \text{ if } i=j\\
-3\alpha_i\xb_j^2 & \text{ if } i = j + 1\\
0 & \text { otherwise. } 
\end{array}
\right. 
\end{equation*}
Thus, for $i=1$, $L_i=0$, and for all $i\geq 2$, $L_i = 30\alpha_i$ if we assume $\{\xb^k\}$ remains bounded in $[-1,1]^n$. 
We employ the same three modes of generation (imbalanced, progressive, and balanced) as in \Cref{sec:genrosenbr} and generate random problem instances in the same way. We once again let $n=p=16$. 

\subsection{Tested variants of \Cref{alg:dfotr}}
We generated several variants of \Cref{alg:dfotr}.
When we assume access to full gradient information (i.e., when we work with logistic loss functions in these experiments), we refer to the corresponding variant of our method as \texttt{SAM-FO} (stochastic average models -- first order) and use the models \cref{eq:first_order}. 
When we do not assume access to gradient information and work with least-squares minimization (i.e., when we work with generalized Rosenbrock and cube functions in these experiments), we extend \texttt{POUNDERS} and refer to the corresponding variant of our method as \texttt{SAM-POUNDERS}, using the models \cref{eq:zero_gauss_newton}. 

For both variants, \texttt{SAM-FO} and \texttt{SAM-POUNDERS}, we split each into two modes of generating $I_k$ and $J_k$ in each iteration. 
For the \emph{uniform} version of a variant, when computing the subset $I_k$ or $J_k$, we take a given computational resource size $r$ and generate a uniform random sample without replacement of size $r$ from $\{1,\dots,p\}$ .
For the \emph{dynamic} version of a variant, when computing the subset $I_k$ or $J_k$, we implement \Cref{alg:dynamic_batchsize}.
The intention of implementing uniform variants is to demonstrate that sampling according to our prescribed distributions is empirically superior to uniform random sampling, which is, arguably, the first naive mode of sampling someone might try.

For all variants of our method, we set the trust-region parameters of Ccref{alg:dfotr} to fairly common settings, that is, 
$\Delta_0=1$, $\Delta_{\max}=1000$, $\gamma=2$, and $\eta_1=0.1$. 
Although not theoretically supported, we set $\eta_2=\infty$, so that step acceptance in \Cref{alg:dfotr} is effectively  based only on the ratio test defined by $\eta_1$. 
For the parameters of \Cref{alg:dynamic_batchsize} in the dynamic variant of our method, we chose $\pi = 0.99$ and  $C=\sum_{i} L_i$. 
For tests of \texttt{SAM-POUNDERS}, all model-building subroutines and default parameters specific to those subroutines were lifted directly from \texttt{POUNDERS}.
We particularly note that, as an internal algorithmic parameter in \texttt{POUNDERS}, when computing \cref{eq:global1_zogn} and \cref{eq:global2_zogn} for use in \Cref{alg:dynamic_batchsize}, we use the upper bound $\|V_{Y_i}^{-1}\|\leq \min\{\sqrt{n}, 10\}$.  

\subsection{Visualizing the method}
Focusing on logistic problems, we demonstrate in \Cref{fig:visualize_fo} and \Cref{fig:visualize_pounders} a single run of \Cref{alg:dfotr} with dynamic batch sizes and computational resource size $r = 1$.

\begin{figure}[h!]
\centering
\includegraphics[width=.49\textwidth]{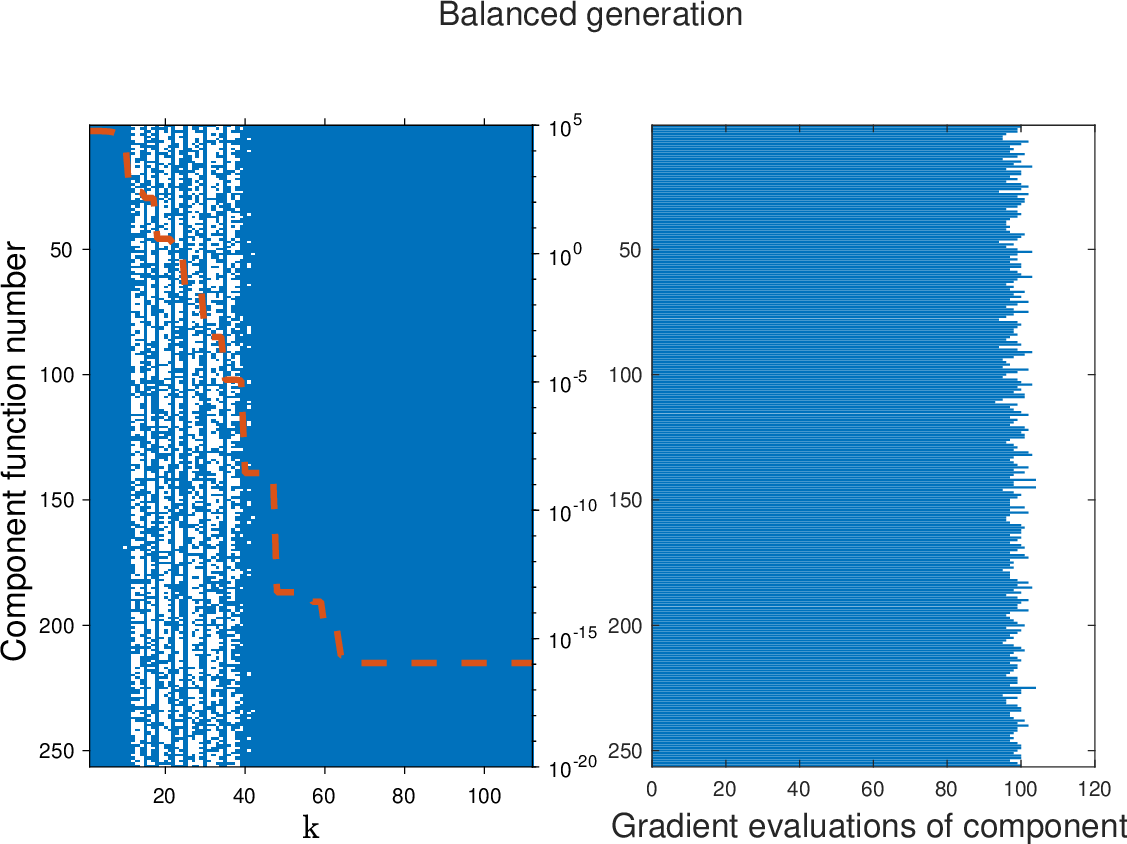} \includegraphics[width=.49\textwidth]{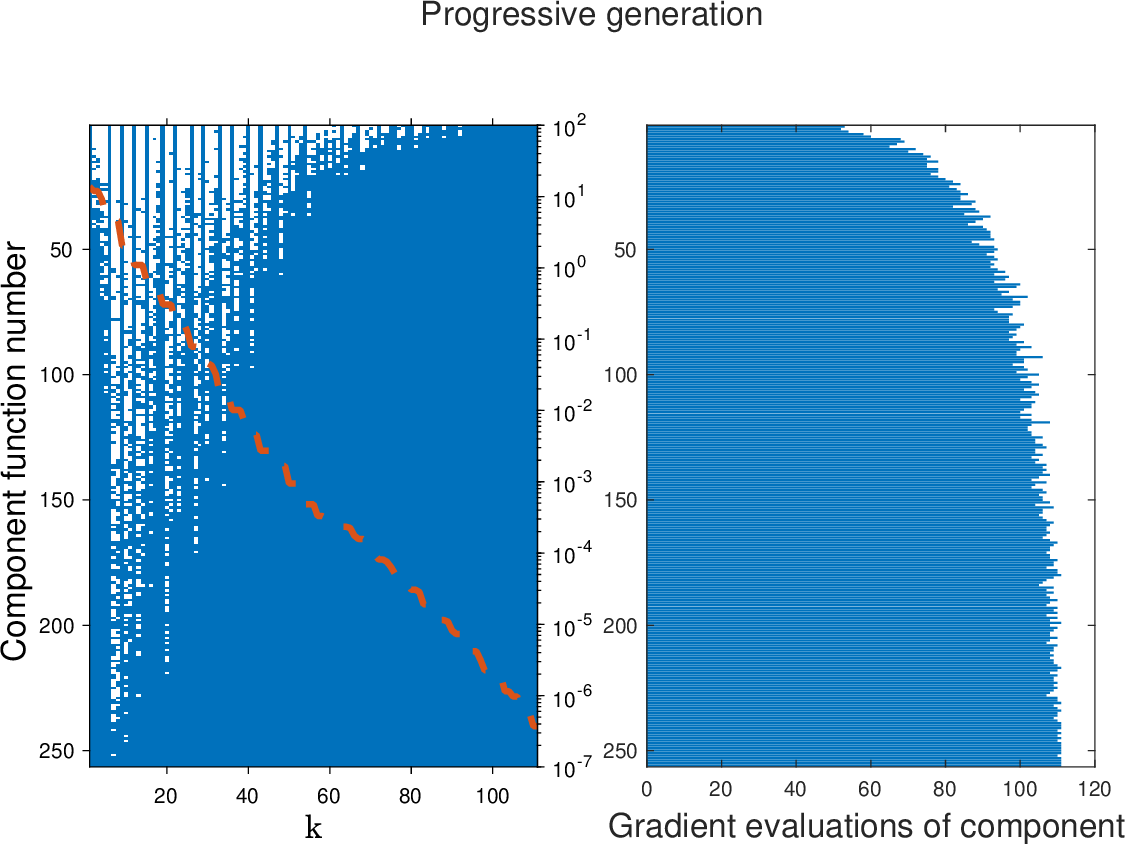}

\includegraphics[width=.49\textwidth]{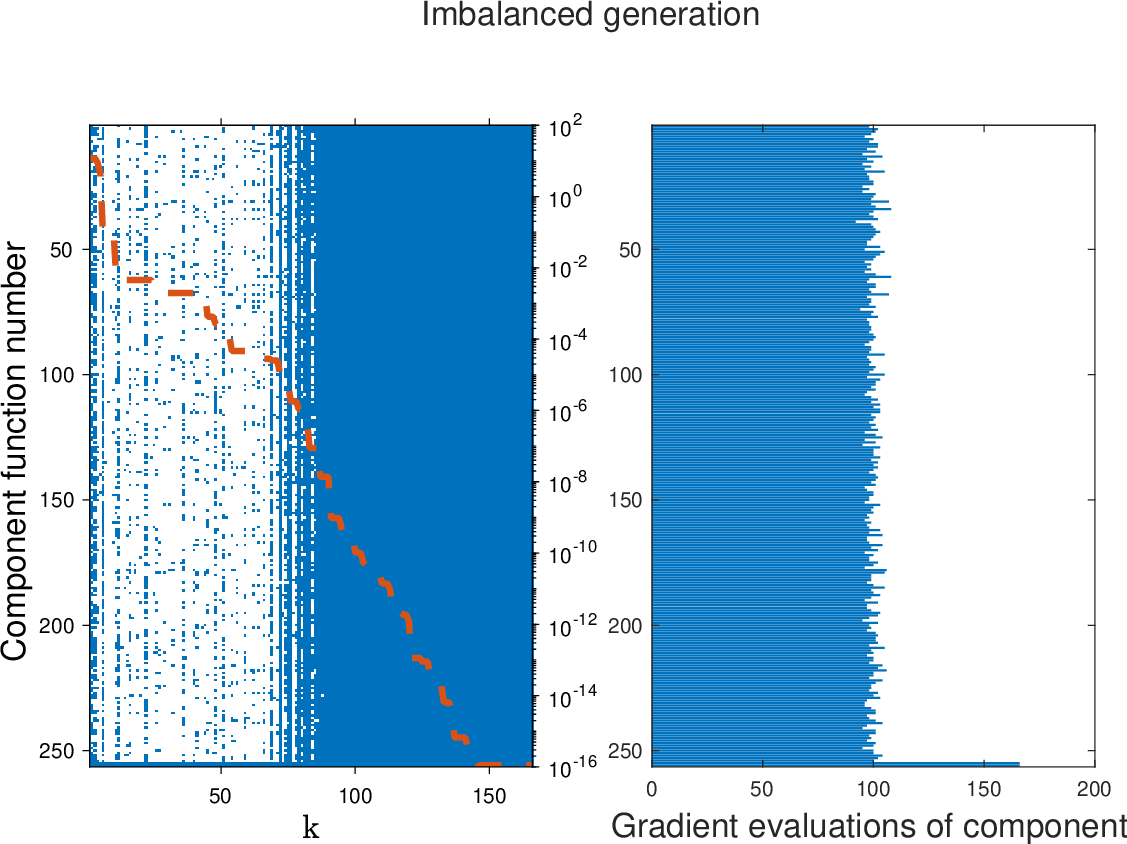}
\caption{Statistics of a single run of \Cref{alg:dfotr} with first-order models \cref{eq:first_order} for each of the three different modes of problem data generation for logistic loss functions. 
In each of the three pairs of figures, the left figure juxtaposes the optimality gap $f(x^k)-f(x^*)$ on top of the sparsity pattern of the evaluations $(F_i(x^k),\nabla F_i(x^k))$ performed in the $k$th point queried by the algorithm.
The histogram in the right figure of each pair illustrates a sum of the corresponding sparsity pattern, namely, the total number of ($F_i(x),\nabla F_i(x)$) evaluations performed. 
\label{fig:visualize_fo}}
\end{figure}

In \Cref{fig:visualize_fo} we illustrate the first-order method on a logistic loss function defined by a single realization of problem data $\{(\ab_{x,i}, \ab_{y,i})\}_{i=1}^p$ under each of the three modes of generation. 
We note in \Cref{fig:visualize_fo} that, as should be expected by our error bounds \cref{eq:first_order} (see \cref{eq:global1_fo}, \cref{eq:global2_fo}), the distribution of the sampled function/gradient evaluations $(F_i(\xb^k),\nabla F_i(\xb^k))$ appears to be proportional to the distribution of the Lipschitz constants $L_i$. 
Moreover, we notice that in the balanced setting, the method tends to sample fairly densely on most iterations, whereas in the progressive and imbalanced settings, the method is remarkably sparser in sampling. 

\begin{figure}[h!]
\centering
\includegraphics[width=.49\textwidth]{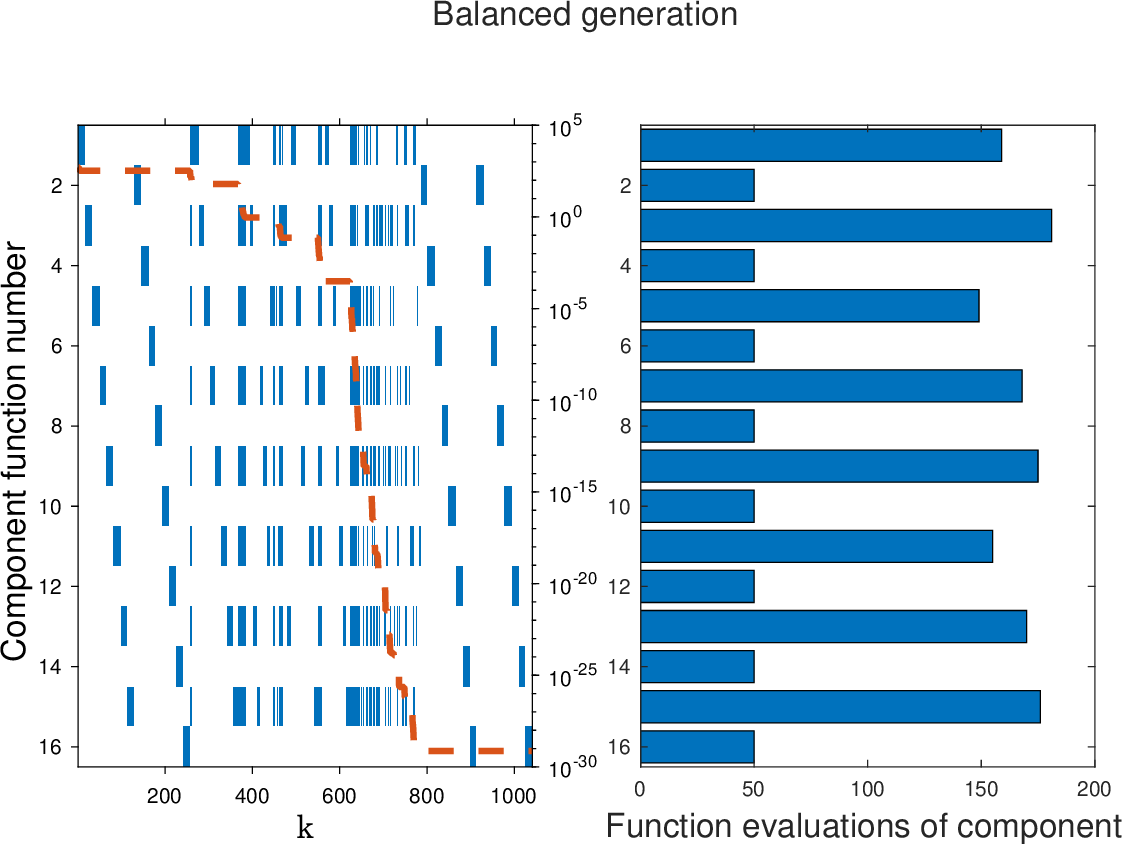} \includegraphics[width=.49\textwidth]{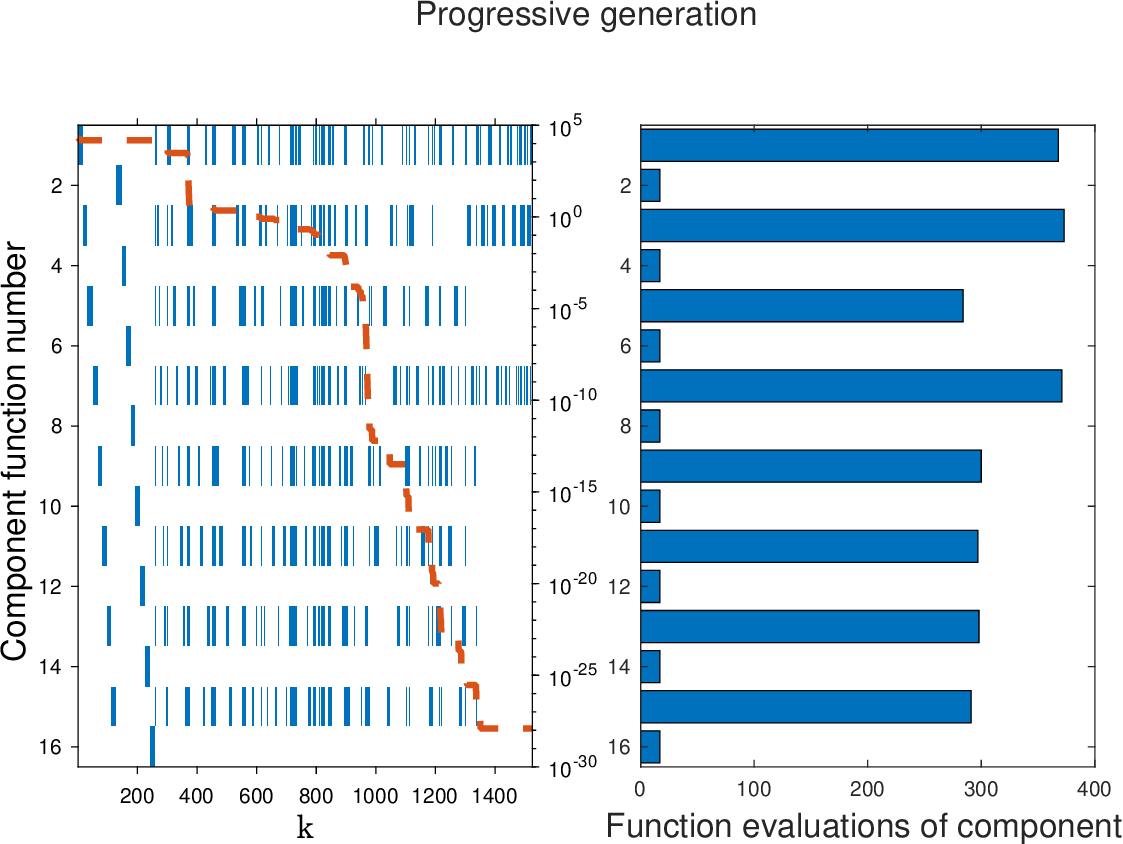}

\includegraphics[width=.49\textwidth]{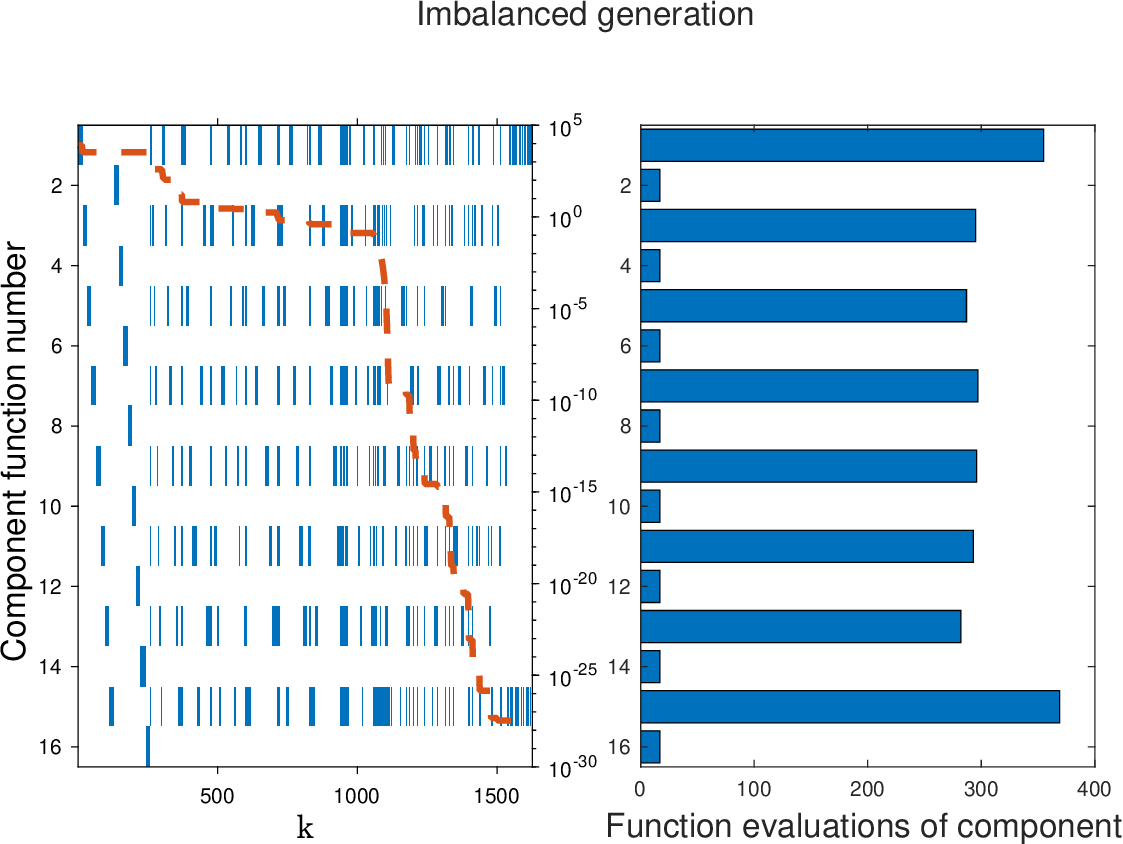}
\caption{Statistics of a single run of \Cref{alg:dfotr} using \texttt{POUNDERS} routines for model building for each of the three different modes of problem data generation for the generalized Rosenbrock function.
The interpretation of the plots is the same as in \Cref{fig:visualize_fo} except that we now  perform only function evaluations (as opposed to gradient evaluations) at a queried point $x^k$. 
\label{fig:visualize_pounders}}
\end{figure}

In \Cref{fig:visualize_pounders}, we illustrate the \texttt{POUNDERS} extension on the generalized Rosenbrock function with parameters $\alpha_i$ defined by each of the three modes of generation. 
As expected from the error bounds for \cref{eq:zero_gauss_newton} (that is, \cref{eq:global1_zogn} and \cref{eq:global2_zogn}), we see that the even-numbered component functions with $L_i=0$ are  sampled only at the beginning of the algorithm (to construct an initial model) and sometimes at the end of the algorithm (due to criticality checks performed by \texttt{POUNDERS}).
However, likely because  of the additional multiplicative presence of $|f_i(\cb_i^k)|$ in the error bounds, it is not the case---as in \Cref{fig:visualize_fo}---that the remaining sampling of the odd-numbered component functions is sampled proportionally to their corresponding Lipschitz constants $L_i$. 

\subsection{Comparing \texttt{SAM-FO} with \texttt{SAG}}\label{sec:sag_experiments}
Because \texttt{SAM-FO} and \texttt{SAG} employ essentially the same average model, it is worth beginning our experiments with a quick comparison of the two methods. 
Because \texttt{SAM-FO} has globalization via a trust region and can employ knowledge of Lipschitz constants of $L_i$, the most appropriate comparison is with the method referred to as \texttt{SAG-LS (Lipschitz)} in \cite{Schmidt2013}, which employs a Lipschitz line search for globalization and can employ knowledge of $L_i$. 
We used the implementation of \texttt{SAG-LS (Lipschitz)} associated with \cite{Schmidt2013}.\footnote{Code taken from \url{https://www.cs.ubc.ca/~schmidtm/Software/SAG.html}} 
Because \texttt{SAG-LS (Lipschitz)} effectively only updates one model at a time, we choose to compare \texttt{SAG-LS (Lipschitz)} only with \texttt{SAM-FO} with a computational resource size $r=1$ and the dynamic mode of generating $I_k,J_k$ (since the uniform mode disregards $L_i$).  

Results are shown in \Cref{fig:sag_experiments}. 
Throughout these, we use the term \emph{effective data passes} to refer to the number of component function evaluations performed, divided by $p$, the total number of component functions. 
With this convention, a deterministic method that evaluates all $p$ component function evaluations in every iteration performs exactly one effective data pass per iteration.
Although effective data passes are certainly related to the notion of an epoch in machine learning literature, they differ in that an epoch typically involves some shuffling so that all data points (or, in our setting, component function evaluations) are touched once per effective data pass.
This notion of equal touching is not applicable to our randomized methods. 

\begin{figure}[h!]
\centering
\includegraphics[width=.99\textwidth]{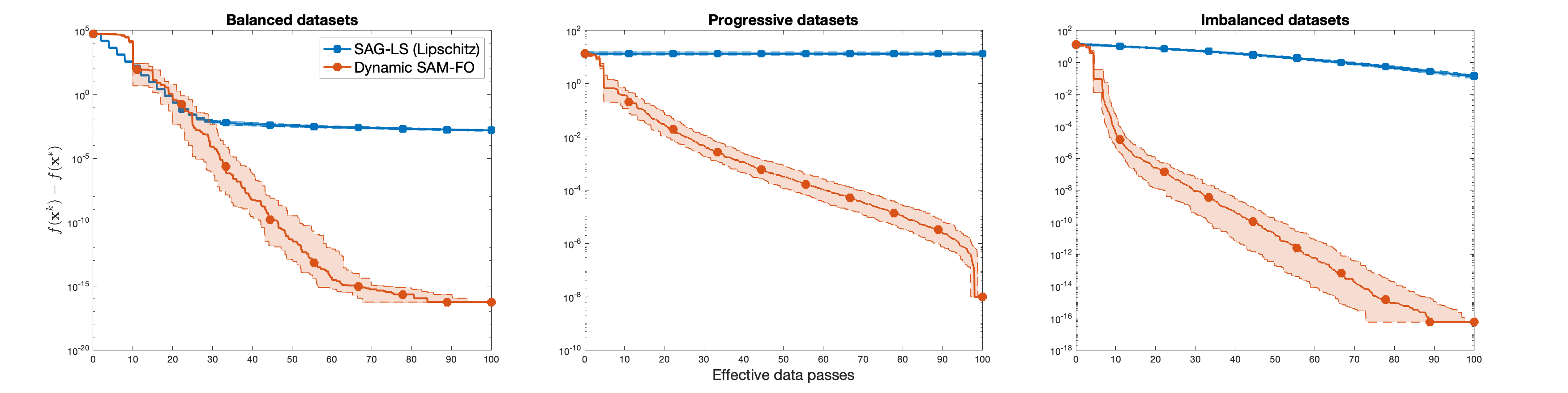}
\caption{\label{fig:sag_experiments} Comparing \texttt{SAG-LS (Lipschitz)} with \texttt{SAM-FO} with dynamic batch sizes on logistic loss problems with \textbf{left)} balanced data generation, 
\textbf{center)} progressive data generation, and \textbf{right)} imbalanced data generation. Solid lines and markers denote median performance across the 90 problems (30 random datasets $\times$ 3 random seeds per dataset), while the outer bands denote $25^{th}--75^{th}$ percentile performance. We note that on the $x$-axis, $f(\xb^k)-f(\xb^*)$ is an appropriate metric because these logistic loss test problems are strongly convex. }
\end{figure}

We note in \Cref{fig:sag_experiments} that while \texttt{SAM-FO} clearly outperforms the out-of-the-box version of \texttt{SAG-LS (Lipschitz)} in these experiments when measured in component function evaluations, this does \emph{not} suggest that \texttt{SAM-FO} would be a preferable method to \texttt{SAG-LS (Lipschitz)} in supervised machine learning (finite-sum minimization) problems. 
Clearly, \texttt{SAM-FO} involves nontrivial computational and storage overhead in maintaining separate models and computing error bounds; and when the number of examples in a dataset is huge (as is often the case in machine learning settings), this overhead might become prohibitive.
Thus, although we can demonstrate that for problems with hundreds (or perhaps thousands) of examples, \texttt{SAM-FO} is the preferable method, we do not want the reader to extrapolate to huge-scale machine learning problems.

\subsection{Comparing uniform SAM with dynamic SAM}\label{sec:uni_vs_dyn}
We now compare both variants of SAM with themselves when generating batches of a fixed computational resource size uniformly at random versus when generating batches according to our suggested \Cref{alg:dynamic_batchsize}.

\subsubsection{\texttt{SAM-FO} on logistic loss problems}
We first illustrate the performance of \texttt{SAM-FO} under these two randomized batch selection schemes on the logistic loss problems, in other words, the same computational setup as in \Cref{sec:sag_experiments}. Results are shown in \Cref{fig:logistic_compare}.

\begin{figure}
 \centering
 \includegraphics[width=.99\textwidth]{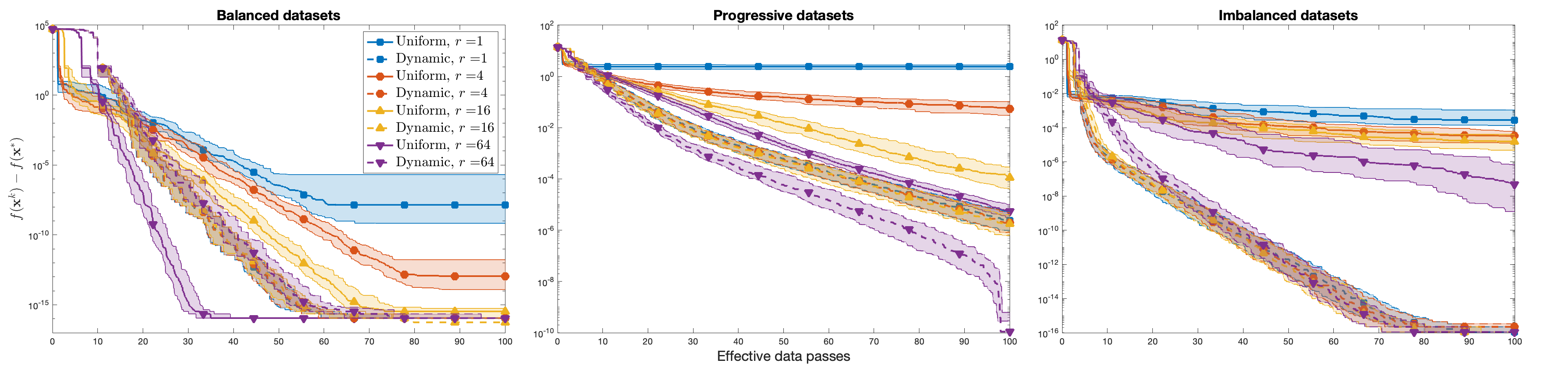}
 \caption{\label{fig:logistic_compare} Comparing the performance of \texttt{SAM-FO} with itself when using uniform generation of batches of a fixed resource-size $r$ versus generating batches according to \Cref{alg:dynamic_batchsize} with parameter $r$. We show results using the same percentile bands as in \Cref{fig:sag_experiments} and separate results by the mode of generating the dataset (balanced, progressive, or imbalanced Lipschitz constants).}
\end{figure}

For the runs using balanced data in \Cref{fig:logistic_compare}, we observe that---perhaps unsurprisingly---for larger values of computational resource size $r$, uniform sampling is marginally better than dynamic sampling. This phenomenon might be explained by the fact that, with all Lipschitz constants roughly the same, the probabilities assigned by \Cref{alg:dynamic_batchsize} are more influenced by the distance between $x^k$ and $y_i^k$ than anything else, and so ``on average'' the updates are nearly cyclic. 
For the progressive datasets in the same figure, we see a more obvious preference for the dynamic variant of \texttt{SAM-FO} across values of $r$.
For the imbalanced datasets in the same figure, we still see the same preference, but we notice that the dynamic variant of \texttt{SAM-FO} hardly loses any performance between $r=1$ and $r=128$; this was  expected because as long as a batch includes the component function with the large Lipschitz constant ``on most iterations,'' then the method should be relatively unaffected by the computational resource size $r$. 

In \Cref{fig:det_compare_logistic}, we display less traditional plots, which we now explain. 
We introduce the notion of \emph{machine size}, which we define as the number of component function evaluations that can be made embarrassingly parallel on a theoretical hardware architecture. We make the simplifying assumption that each component function evaluation requires an equal amount of computational resources to perform, which is approximately correct for the test problems in this paper. 
With the notion of machine size, we can then define the number of \emph{rounds} required by a dynamic SAM variant with resource size parameter $r$ to solve a problem $\pi$ to a defined level of tolerance $\tau$, when run on a machine of a given machine size $\mu$; that is,
$$\begin{array}{l}R_{r,\pi,\tau,\mu} = \\
\left(\# \text{batches of size $r$ required to satisfy $f(\xb^k)-f(\xb^*)\leq \tau$ on problem } \pi\right) \left\lceil\displaystyle\frac{r}{\mu}\right\rceil.
\end{array}$$
The number of rounds $R_{r,\pi,\tau,\mu}$ is an idealized estimation of wall-clock time for estimating scalability. 
For example, for the logistic loss problem with $p=256$, then 256 function evaluations could be done in the same amount of time required by a single-component function evaluation given a machine size of $\mu=256$. 
On the other extreme, if $\mu=1$, then 256 function evaluations would require the amount of time required by 256 single-component function evaluations (that is, they would have to be executed serially). 

\begin{figure}[h!]
\centering
 \includegraphics[width=.99\textwidth]{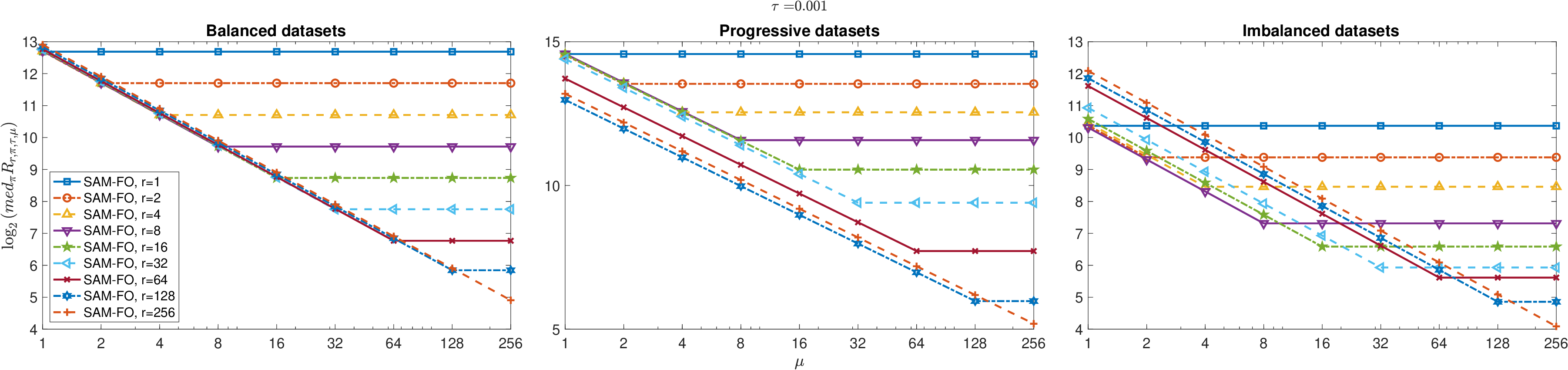}
 
  \includegraphics[width=.99\textwidth]{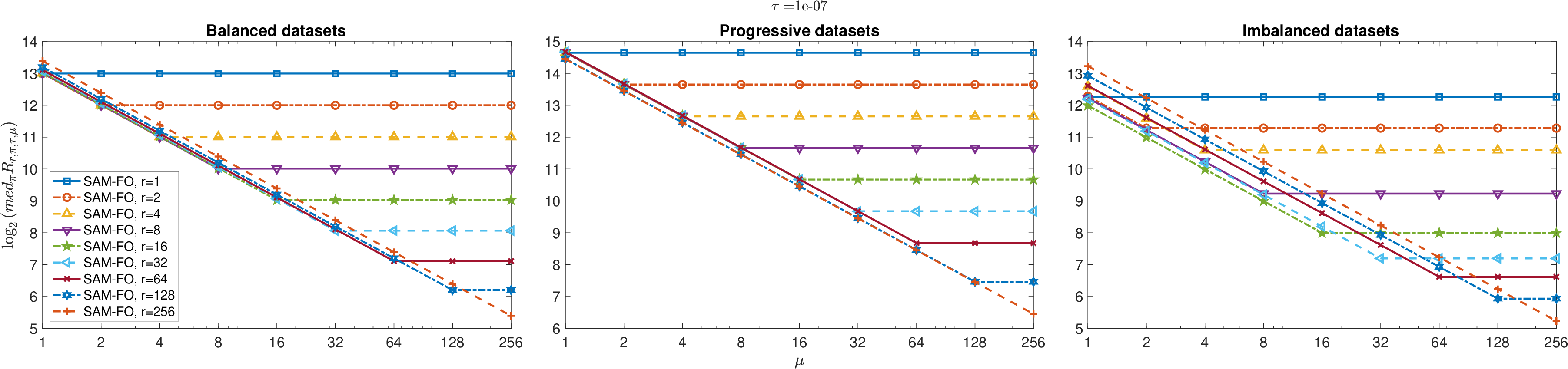}
\caption{\label{fig:det_compare_logistic} For each mode of generating random tested logistic loss problems, we show the median, over the problems $\pi$, of $\log_2(R_{r,\pi,\tau,\mu})$. The top row displays results for a convergence tolerance $\tau=10^{-3}$, and the bottom row displays results for the tighter convergence tolerance $\tau=10^{-7}$. }
\end{figure}

We remark in \Cref{fig:det_compare_logistic} that when using balanced datasets, for all tested values of $r$, \texttt{SAM-FO}-$r$ requires roughly $90\%$ of the computational resource use of the deterministic method in its median performance whenever $r\geq \mu$ for the tighter tolerance $\tau=10^{-7}$. 
More starkly, when using imbalanced datasets, \texttt{SAM-FO} uses between $40\%$ and $80\%$ of the computational resources required by the deterministic method when $\mu\leq r \leq128$.
Perhaps as expected, results are less satisfactory for the progressive datasets, which are arguably the hardest of these problems. 
Even for the progressive datasets, however, the regret is not unbearably high in the tighter convergence tolerance, with \texttt{SAM-FO}-$r$ requiring $130\%$ of the computational resources required by the deterministic method when $\mu \leq r \leq 32$, and the randomized and deterministic method roughly breaking even for $\mu \leq r\in\{64,128\}$. 

\subsubsection{\texttt{SAM-POUNDERS} on Rosenbrock problems}

We now illustrate the performance of \texttt{SAM-POUNDERS} under the two randomized batch selection schemes on the Rosenbrock problems.
Results are shown in \Cref{fig:rosenbrock_compare}. 

\begin{figure}
 \centering
 \includegraphics[width=.99\textwidth]{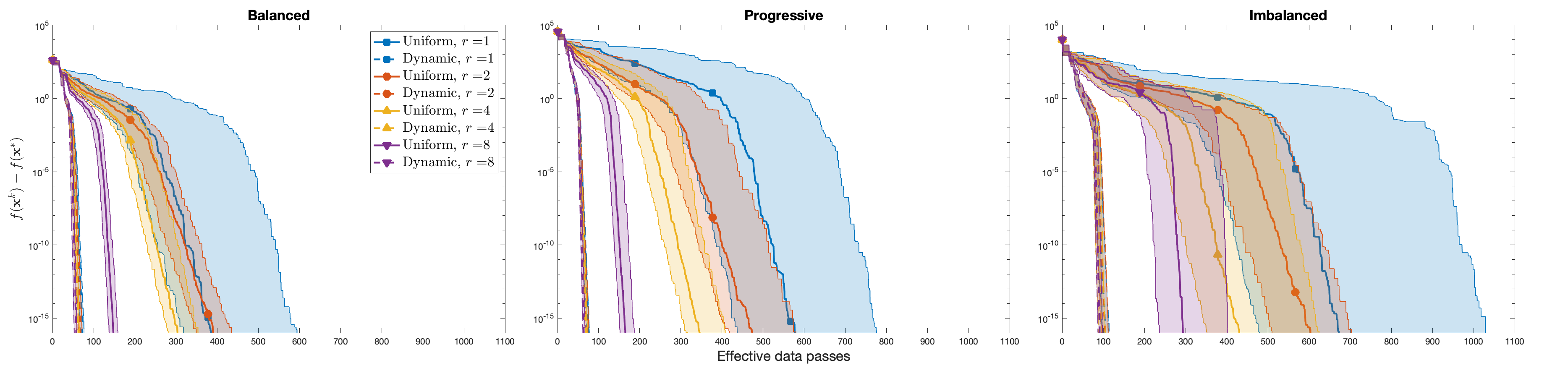}
 \caption{\label{fig:rosenbrock_compare} Comparing the performance of \texttt{SAM-POUNDERS} with itself when using uniform generation of batches of a fixed resource  size $r$ versus generating batches according to \Cref{alg:dynamic_batchsize} with the Rosenbrock problems. We verified that no run converged to a local minimum, and hence the optimality gap on the $y$-axis is appropriate. }
\end{figure}

We remark in \Cref{fig:rosenbrock_compare} that there is a generally clear preference for employing the dynamic batch selection method over simple uniform batch selection. 

We again employ the same comparisons of \texttt{SAM-POUNDERS}-$r$ to
\texttt{POUNDERS} as in \Cref{fig:det_compare_logistic} in \Cref{fig:det_compare_rosenbrock}.
\begin{figure}[h!]
\centering
 \includegraphics[width=.99\textwidth]{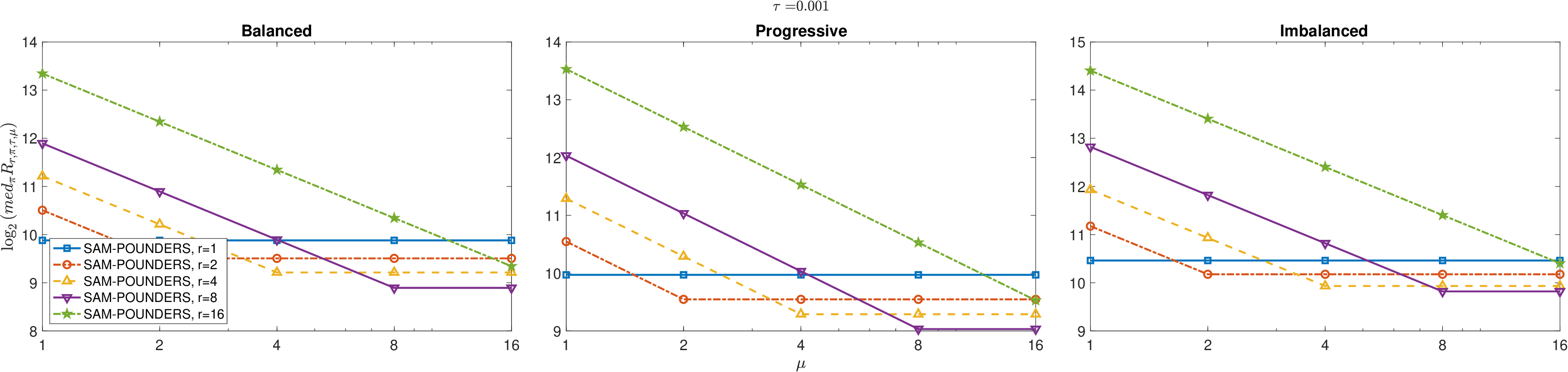}
 
  \includegraphics[width=.99\textwidth]{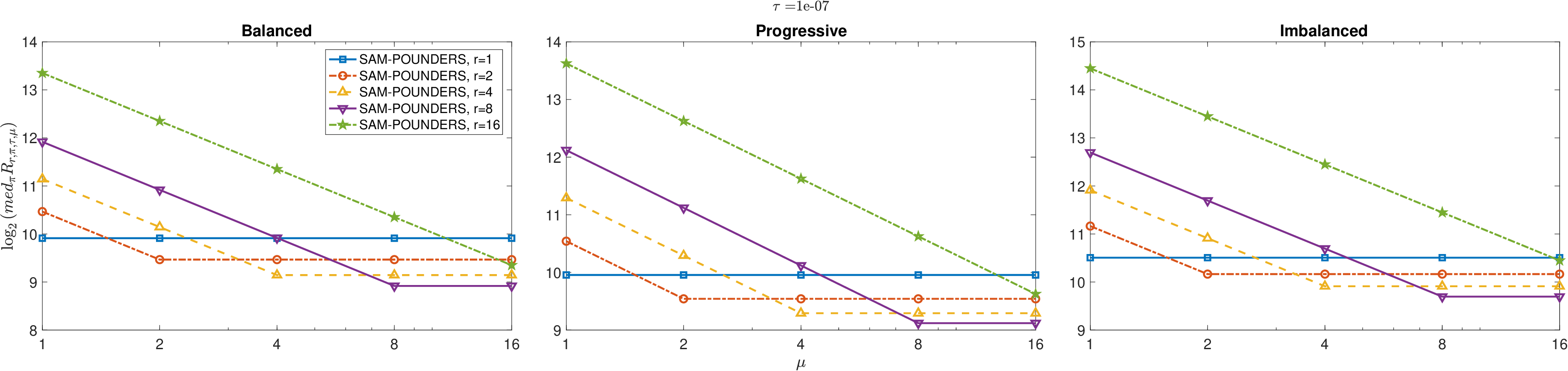}
\caption{\label{fig:det_compare_rosenbrock} For each mode of generating generalized Rosenbrock problems, we show the median, over the problems $\pi$, of $\log_2(R_{r,\pi,\tau,\mu})$. The top row displays results for a convergence tolerance $\tau=10^{-3}$, and the bottom row displays results for the tighter convergence tolerance $\tau=10^{-7}$.}
\end{figure}
We remark that the results shown in \Cref{fig:det_compare_rosenbrock} are  satisfactory, with all problems and values of $r$ giving a strict improvement in computational resource use to attain either $\tau=10^{-1}$-optimality or $\tau=10^{-7}$-optimality over the deterministic method. 

\subsubsection{\texttt{SAM-POUNDERS} on cube problems}

We illustrate the performance of \\
\texttt{SAM-POUNDERS} under the two randomized batch selection schemes on the cube problems.
Results are shown in \Cref{fig:cube_compare}. 

\begin{figure}
 \centering
 \includegraphics[width=.99\textwidth]{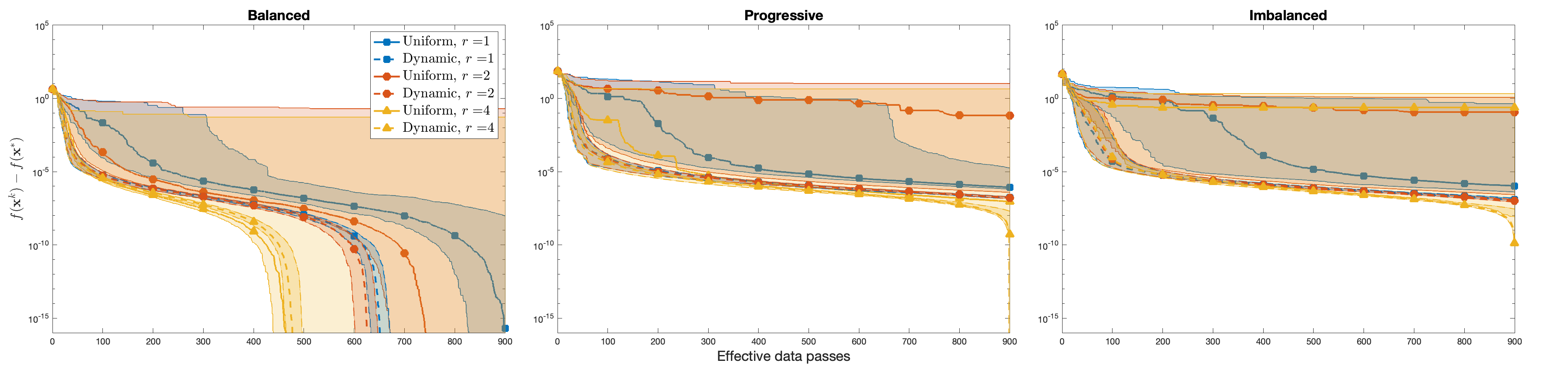}
 \caption{\label{fig:cube_compare} Comparing the performance of \texttt{SAM-POUNDERS} with itself when using uniform generation of batches of a fixed resource-size $r$ versus generating batches according to \Cref{alg:dynamic_batchsize} with the cube problems. We verified that no run converged to a local minimum, and hence the optimality gap on the $y$-axis is appropriate. }
\end{figure}

As in the Rosenbrock tests, we find in \Cref{fig:cube_compare} a preference for employing the dynamic batch selection method over simple uniform batch selection on the cube problems. 

\begin{figure}[h!]
\centering
 \includegraphics[width=.99\textwidth]{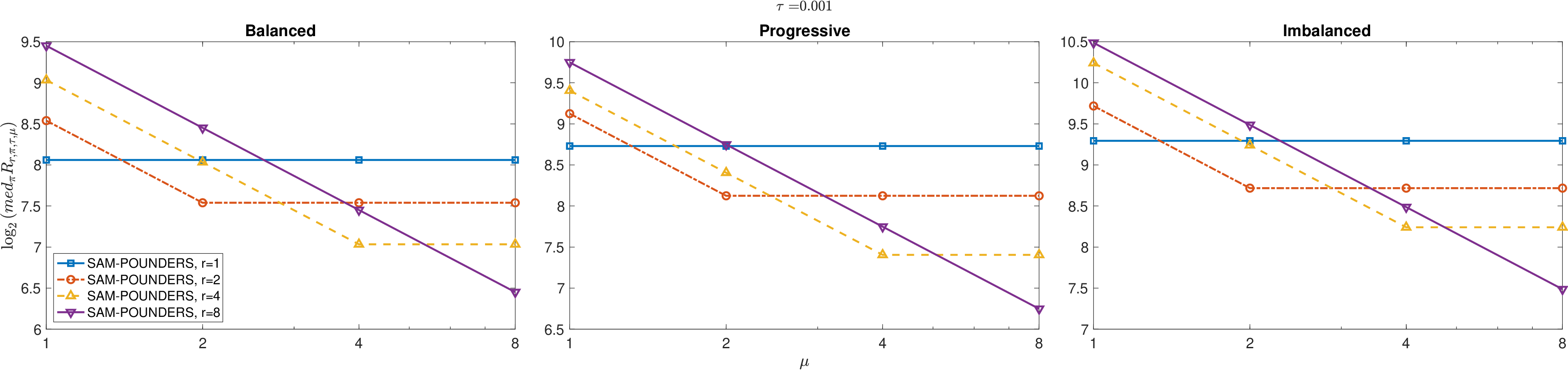}
 
  \includegraphics[width=.99\textwidth]{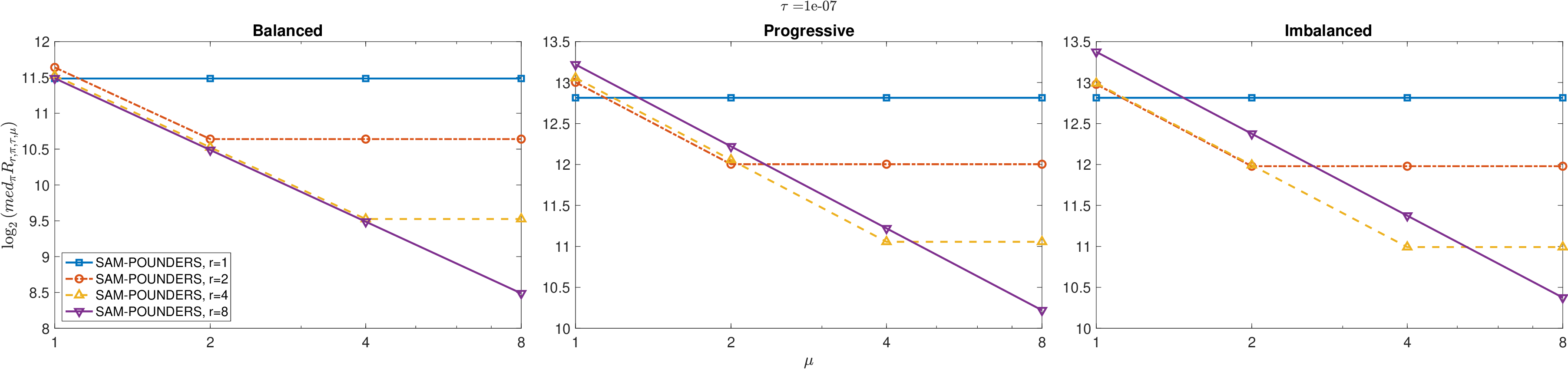}
\caption{\label{fig:det_compare_cube} For each mode of generating cube problems, we show the median, over the problems $\pi$, of $\log_2(R_{r,\pi,\tau,\mu})$. The top row displays results for a convergence tolerance $\tau=10^{-3}$, and the bottom row displays results for the tighter convergence tolerance $\tau=10^{-7}$.}
\end{figure}

In \Cref{fig:det_compare_cube}, we see that compared with the deterministic method, 
for the weaker convergence tolerance ($\tau=10^{-3}$) the dynamic variants of \texttt{SAM-POUNDERS}-$r$ yield better median performance than deterministic \texttt{POUNDERS}. 
However, the situation is  more mixed in the tighter convergence tolerance ($\tau=10^{-7}$).
While \texttt{SAM-POUNDERS}-$r$ exhibits better median performance than does deterministic \texttt{POUNDERS} for all $r$ in the imbalanced generation setting, the same is only true for $r=1$ in the balanced case and $r=1,2$ in the progressive generation setting. 
As in the logistic loss experiments, however, we see that even in the situations where \texttt{SAM-POUNDERS}-$r$ does not outperform the deterministic method, it does not lose by a significant amount, leading to low regret. 

\subsection{Comparison of \texttt{SAM-POUNDERS} with \texttt{POUNDERS}}\label{sec:pounders_compare}
We now illustrate, using the same test problems as in \Cref{sec:uni_vs_dyn}, the comparative performance of \texttt{SAM-POUNDERS} against \texttt{POUNDERS}, the deterministic method on which \texttt{SAM-POUNDERS} was built. 
In theory, \texttt{POUNDERS} could be described as a special case of \texttt{SAM-POUNDERS} that employs a uniform batch of size $p$ in every function evaluation.
However, due to particular considerations in our implementation of \texttt{SAM-POUNDERS} including the parameter $\eta_2$ in \Cref{alg:dfotr}, this recovery shouldn't be expected to happen in practice. 
Thus, in these tests we directly compare \texttt{POUNDERS}\footnote{The current version of \texttt{POUNDERS} is actively maintained at \url{https://github.com/POptUS/IBCDFO}.} with our novel implementation of \texttt{SAM-POUNDERS} with a dynamic batch generation and assumed computational resource size $r=1$.
In \Cref{fig:pounders_rosenbrock_16}, we illustrate such results on $n=p=16$-dimensional Rosenbrock functions with the same modes of generation (balanced, progressive, imbalanced) as described in \Cref{sec:genrosenbr}.
In \Cref{fig:pounders_cube_16}, we illustrate such results on $n=p=16$-dimensional cube functions, again with the same three modes of generation. 
In \Cref{sec:add_num}, we show results for the same experimental setup, but with $n=p=64$. 
In all cases, we see general benefits in terms of the number of component function evaluations by using the dynamic batch generation of \texttt{SAM-POUNDERS}.

\begin{figure}   
    \includegraphics[width=.99\textwidth]{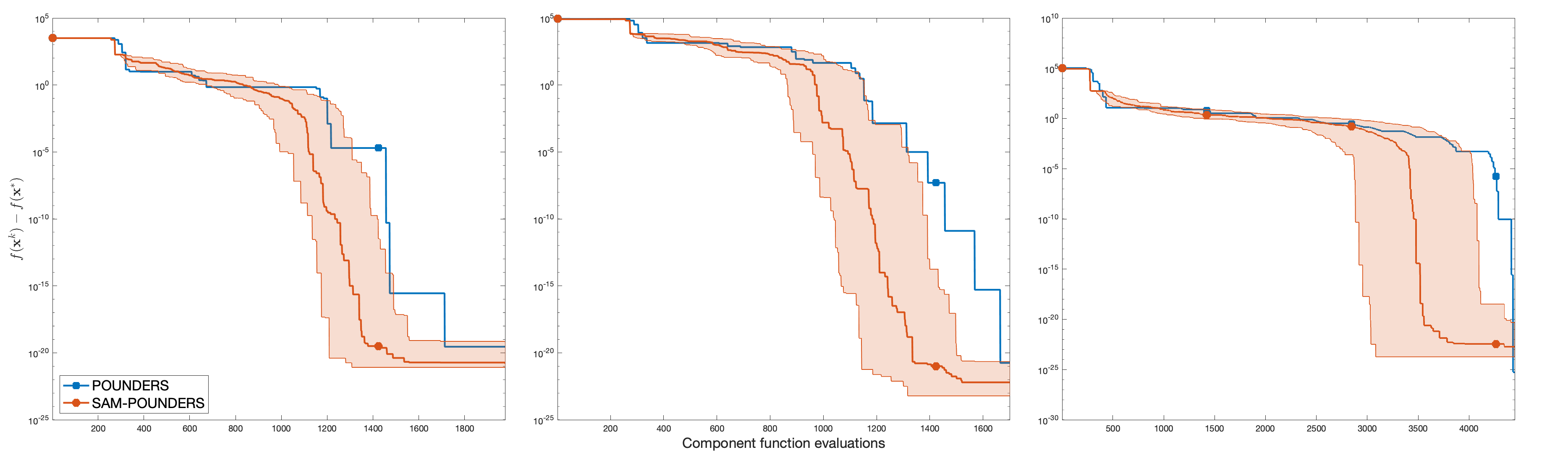}
    \caption{\label{fig:pounders_rosenbrock_16} Comparison of \texttt{POUNDERS} and \texttt{SAM-POUNDERS} on $n=p=16$-dimensional Rosenbrock problems with balanced (left), progressive (center), and imbalanced (right) modes of generation. Both solvers were given a common starting point, and \texttt{SAM-POUNDERS} was run with 30 different random seeds. Median performance is illustrated in the solid line of \texttt{SAM-POUNDERS}, with $10-90\%$ percentile bands illustrated in the transparent filled region.}
\end{figure}

\begin{figure}
    \includegraphics[width=.99\textwidth]{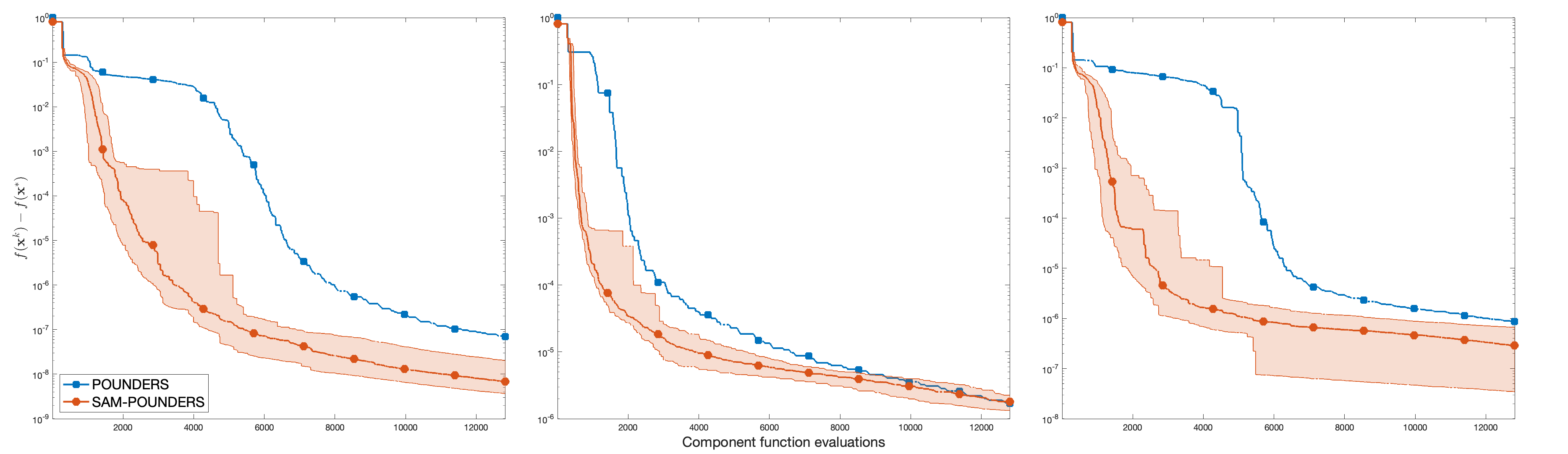}
    \caption{\label{fig:pounders_cube_16} Same as \Cref{fig:pounders_rosenbrock_16}, but with cube problems.} 
\end{figure}

\subsection{Performance of SAM when not assuming knowledge of $L_i$}\label{sec:nolip}
Until now, we have given the SAM methods access to global Lipschitz constants $L_i$ in order to compute the various bounds employed in the dynamic batch selection variants.
However, assuming access to $L_i$ is typically an impractical assumption: while it is practical in the logistic loss problems, it is virtually never practical in any setting of derivative-free optimization. 

Thus, we experiment with a simple modification to any given SAM method. Rather than assuming a value of $L_i$ at the beginning of the algorithm, we employ estimates $\tilde{L}_i$. We arbitrarily initialize $\tilde{L}_i=1$ for $i=1,\dots,p$.
We select $I_0 = \{1,\dots,p\}$ on the first iteration and additionally select $I_{k_0+1} = \{1,\dots,p\}$, where $k_0$ denotes the first successful iteration of \Cref{alg:dfotr}. 
In other words, we are guaranteed at the start of the algorithm to have computed a model of each component function centered at two distinct points; that is, we will have computed 
$m_i(\xb;\xb^0)$ and $m_i(\xb;\xb^{k_0 - 1})$ for $i=1,\dots,p$. 
Immediately after updating the models indexed by $I_{k_0+1}$, we compute a lower bound on a global Lipschitz constant via the secant
$$\tilde{L}_i \gets \displaystyle\frac{\|\nabla m_i(\xb;\xb^{k_0+1}) - \nabla m_i(\xb;\xb^0)\|}{\|\xb^{k_0+1}-\xb^0\|}.$$
Over the remainder of the algorithm, on each iteration $k$ in which $i\in I_k$, 
and immediately after updating the $i$th model, we update
$$\tilde{L}_i \gets \max\left\{\tilde{L}_i, \displaystyle\frac{\|\nabla m_i(\xb;\xb^{k}) - \nabla m_i(\xb;\cb^k_i)\|}{\|\xb^{k}-\cb^k_i\|}\right\},$$
provided that $\xb^k \neq \cb^k_i$. 

\begin{figure}[h!]
 \centering
 \includegraphics[width=.99\textwidth]{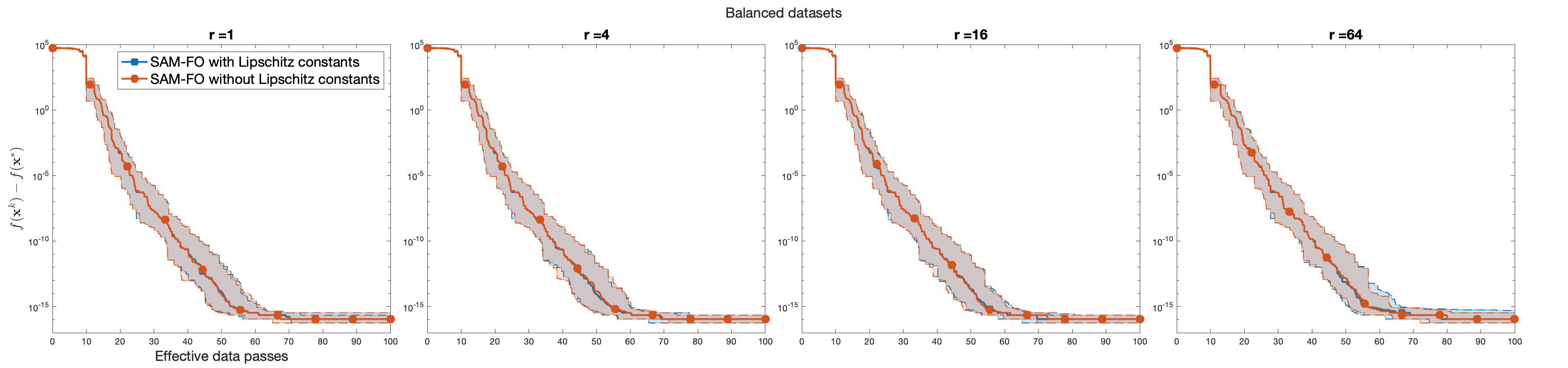}
 
 \includegraphics[width=.99\textwidth]{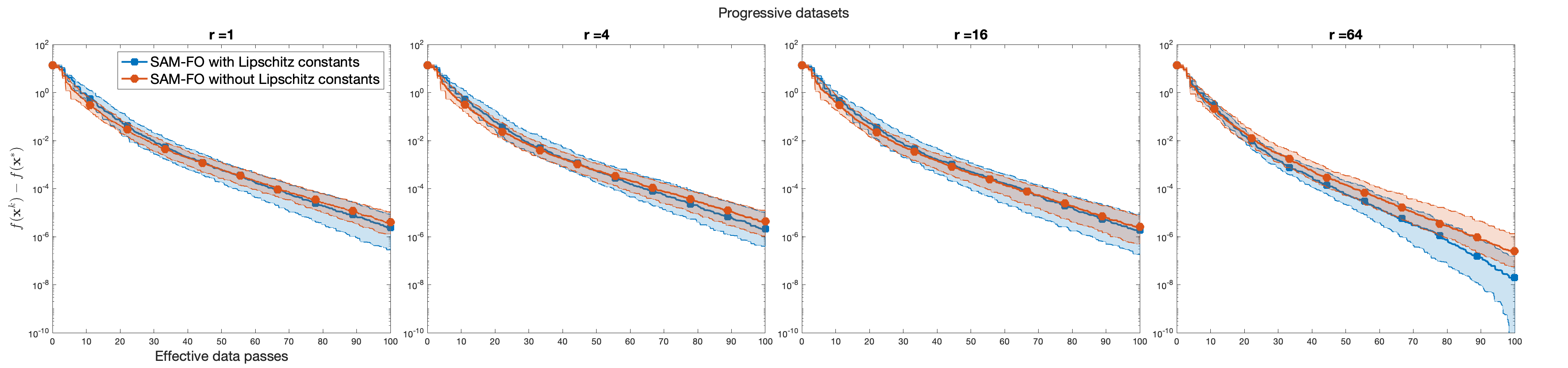}
 
 \includegraphics[width=.99\textwidth]{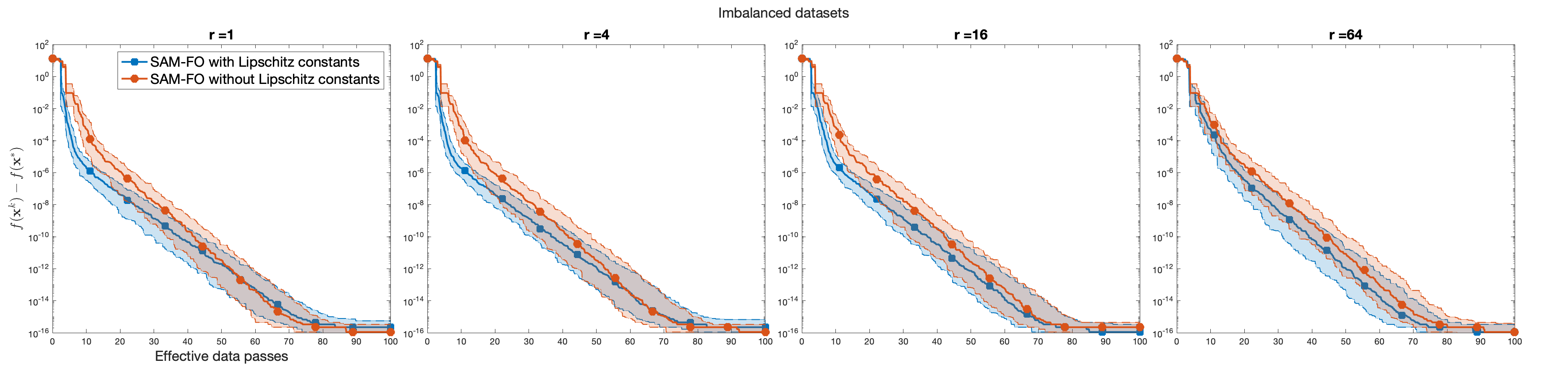}
 \caption{\label{fig:lip_nolip_logistic} Comparing the performance of \texttt{SAM-FO} with dynamic batch selection and resource size $r$ with itself when provided with Lipschitz constants and when using our proposed scheme to dynamically adjust the Lipschitz constant estimates. We show results using the same percentile bands as in \Cref{fig:sag_experiments} and again separate results by the mode of generating the dataset (balanced, progressive, or imbalanced Lipschitz constants).}
\end{figure}

In \Cref{fig:lip_nolip_logistic} we first illustrate results for the logistic loss problems. Remarkably, there is  little difference in performance in the balanced dataset case.
When we move to progressive datasets, it is remarkable that in median performance, and with the possible exception of $r=64$, the performance of \texttt{SAM-FO}-$r$ is typically better without assuming Lipschitz constants than with assuming them explicitly. This is not completely bizarre, however, as the global Lipschitz constants $L_i$ provided to any method are necessarily worst-case upper bounds on local Lipschitz constants, and it is generally unlikely that a method will encounter the upper bound---hence, the component functions with larger $L_i$ values (i.e., those with indices closer to $p$) may be updated more frequently than they need to be in the randomized methods, leading to slowdowns when $L_i$ is available. 
In the imbalanced datasets test, we see that \texttt{SAM-FO}-$r$ indeed loses some performance when using estimates $\tilde{L}_i$, but not drastically. 
A reasonable explanation for this phenomenon is that the overapproximation provided by $L_i$ that proved to be potentially problematic in the progressive dataset is in fact helpful in the imbalanced datasets---indeed we want to be updating the $p$th model with much higher frequency relative to the other $p-1$ components over the run of the algorithm, and so a relatively large value of $L_i$ provided by a global Lipschitz constant will certainly force this to happen. 

For the sake of space in the main body of text, we move the results of these experiments for the remaining two classes of problems with \texttt{SAM-POUNDERS} (generalized Rosenbrock and cube) to \Cref{sec:add_num}.
It suffices to say that in all cases, there is some loss in performance by replacing $L_i$ with the coarse estimate $\tilde{L}_i$, but the loss in most cases is acceptable. 

\section{Conclusion}
In this work, we have proposed a stochastic average model methodology for subsampling component functions in finite-sum optimization.
We specialize this methodology to extend the model-based derivative-free optimization solver \texttt{POUNDERS} to \texttt{SAM-POUNDERS}.
Our preliminary numerical results were designed to synthesize settings where the gradient Lipschitz constants of individual component functions were characteristically different.
Overall, we found that a version of \texttt{SAM-POUNDERS} that dynamically selects batches (and a batch size) according to coarse upper bounds on estimated changes in the function performs exceedingly well, and most importantly, usually outperforms its deterministic counterpart, \texttt{SAM-POUNDERS}.
We anticipate that \texttt{SAM-POUNDERS} will have immediate practical use cases in nuclear model calibration, but more generally, any setting in derivative-free optimization where the objective is expressed as a finite sum, and component function evaluations can be readily parallelized.

\section*{Acknowledgments}
The authors are grateful to Yong Xie for early discussions that inspired this work
. We are also grateful to two anonymous referees who improved the presentation of the paper, and especially grateful to one referee who provided an interesting discussion of the connections between \cref{eq:prob_fl} and probabilistic full-linearity.

\section*{Declarations}
\textbf{Funding} This work was supported in part by the U.S.~Department of Energy, Office of Science, Office of Advanced Scientific Computing Research Applied Mathematics and SciDAC programs under Contract Nos.~DE-AC02-06CH11357 and DE-AC02-05CH11231.

\textbf{Conflicts of interest} The authors declare no competing interests beyond the stated funding. 

\textbf{Data availability} The code to perform the experiments in this paper is available in a public repository, \url{https://github.com/mmenickelly/sampounders/}.

\begin{appendix}
\section{Statement of \Cref{alg:deville}}\label{sec:deville} 
The algorithm of \cite{chen2000general} essentially amounts to finding the solution $\tilde{\bpi}\in\Reals^p$ to the equation
\begin{equation}\label{eq:recursive_eqn}
\Psi_b(\tilde{\bpi}) = \bpi,
\end{equation}
where in our notation, $b$ is the desired batch size, and $\bpi$ are the desired inclusion probabilities. 
$\Psi_b$ is defined recursively and entrywise via 
\begin{equation}\label{eq:recursive}
\Psi_{0}(\tilde{\bpi}) = \zerob_{p}, \quad 
[\Psi_k(\tilde{\bpi})]_i = k\displaystyle\frac{\displaystyle\frac{\tilde{\pi}_i}{1-\tilde{\pi}_i}\left(1-[\Psi_{k-1}(\tilde{\bpi})]_i\right)}
{\displaystyle\sum_{j=1}^p \frac{\tilde{\pi}_j}{1-\tilde{\pi}_j}\left(1-[\Psi_{k-1}(\tilde{\bpi})]_j\right)}, \quad k = 1,2,\dots,b.
\end{equation}

\begin{algorithm}[!h]
\caption{Transforming $\bpi$ into $\tilde{\bpi}$} 
\label{alg:deville}
\textbf{Input: } Batch size $b>0$, probability vector $\bpi\in\Reals^p$.\\
Initialize $\tilde{\bpi}\gets \zerob_p$.\\
$A\gets\{i: \pi_i=1\}, B\gets \{1,2,\dots,p\}\setminus A, \tilde{\bpi}_A \gets \bpi_A$.\\
\If{$|A| = b$}{\textbf{return} $\tilde{\bpi}$}
Compute $\tilde{\pi}_B^*$ such that $\Psi_{b-|A|}(\tilde{\bpi}_B^*) = \bpi_B.$\\
$\blambda\gets \log\left(\displaystyle\frac{\tilde{\bpi}_B^*}{\oneb_{|B|}-\tilde{\bpi}_B^*} \right)$ ($\log$ and division interpreted entrywise)\\
Compute scalar $c$ such that
\begin{equation}
    \label{eq:scalar_c} 
    b - |A| = \displaystyle\sum_{i\in B} \frac{\exp\left(\lambda_i + c\right)}{1+ \exp\left(\lambda_i + c\right)}.
\end{equation}

$\tilde{\bpi}_B^* \gets \displaystyle\frac{\exp(\blambda + c\oneb_{|B|}) }{\oneb_{|B|} + \exp(\blambda + c\oneb_{|B|})}$ ($\exp$ and division interpreted entrywise) \\
$\tilde{\bpi}_B \gets \tilde{\bpi}_B^*$.\\
\textbf{return } $\tilde\bpi$\\
\end{algorithm}
\Cref{alg:deville} simply solves \cref{eq:recursive_eqn} after preprocessing any $\tilde{\pi}_i=1$, a clearly necessary step in light of \cref{eq:recursive}. 
In our experiments, we used a basic implementation of Newton's method with a backtracking line search for the solution of \cref{eq:recursive_eqn}.
As mentioned in the main body of the text, we then do some postprocessing to force
$$\displaystyle\sum_{i=1}^p \tilde{\pi}_i^k = b.$$
Although technical, the reasoning for why this postprocessing is valid is due to the fact that Poisson sampling is a special case of exponential sampling with a nontrivial affine subspace invariant under the exponential probability denisty function; this is is well-explained in, for instance, \cite{tille2006sampling}[Section 5.6.3]. 
In our experiments, we also solve \eqref{eq:scalar_c} by a basic implementation of Newton's method with a backtracking line search.

 %
    

\section{Proof of \Cref{thm:dfo_bound}}\label{sec:thm1}
\begin{proof}
 We will first derive a bound on $\|\nabla F_i(\xb) - \gb_i^k(\cb^k_i;\delta_i)\|$.
 By \Cref{ass:Y}, $\hat{V}_{Y_i}(\xb;\cb^k_i)$ is invertible, regardless of $\xb$ and $\delta_i>0$. 
 Setting up the linear interpolation system, we see that
 $$(\vb^j-\cb_i^k)^\top \gb_i^k(\cb^k_i;\delta_i) = F_i(\vb^j) - F_i(\cb^k_i), \quad j=1,\dots,n.$$
 By the mean value theorem, this right-hand side is also equal to
 $$F_i(\vb^j) - F_i(\cb^k_i) = \displaystyle\int_0^1 (\vb^j - \cb^k_i)^\top \nabla F_i(\cb^k_i + t(\vb^j-\cb^k_i))\mathrm{d}t, \quad j=1,\dots,n,$$
 and so, by \Cref{ass:f},
 $$(\vb^j-\cb^k_k)^\top(\nabla F_i(\cb^k_i)-\gb_i^k(\cb^k_i;\delta_i)) \leq \frac{L_i}{2}\|\vb^j - \cb^k_i\|^2, \quad j=1,\dots,n.$$
 Combining these $n$ inequalities, noting that $\|\vb^j - \cb^k_i\|^2\leq \delta_i$ for $j=1,\dots,n$, and recalling the definition of $\hat{V}_{Y_i}(\xb;\cb^k_i)$ in \cref{eq:vhatbound},
 $$\|\hat{V}_{Y_i}(\xb;\cb^k_i)(\nabla F_i(\cb^k_i) - \gb_i^k(\cb^k_i;\delta_i))\| \leq \frac{L_i\sqrt{n}}{2\max\{\delta_i,\|\xb-\cb^k_i\|\}}\delta_i^2,$$
 and so
 \begin{equation}\label{eq:grad_at_center_bound}
 \|\nabla F_i(\cb^k_i) - \gb_i^k(\cb^k_i;\delta_i)\| \leq \frac{L_i\sqrt{n}\|\hat{V}_{Y_i}^{-1}(\xb;\cb^k_i)\|}{2\max\{\delta_i,\|\xb-\cb^k_i\|\}}\delta_i^2.
 \end{equation}
 Thus, for any $\xb \in \Reals^n$,
 \begin{equation}\label{eq:grad_bound}
 \begin{array}{lll}
  \|\nabla F_i(\xb) - \gb_i^k(\cb^k_i;\delta_i)\| & \leq &\|\nabla F_i(\xb) - \nabla F_i(\cb^k_i)\| + \|\nabla F_i(\cb^k_i) - \gb_i^k(\cb^k_i;\delta_i)\|\\
  & \leq & 
 L_i\left(\|\xb-\cb^k_i\|+\frac{\sqrt{n}\|\hat{V}_{Y_i}^{-1}(\xb;\cb^k_i)\|}{2\max\{\delta_i,\|\xb-\cb^k_i\|\}}\delta_i^2 \right).
 \end{array}
 \end{equation}
 Now, by Taylor's theorem,
 $$F_i(\xb) - F_i(\cb^k_i) \leq \nabla F_i(\cb^k_i)^\top (\xb-\cb^k_i) + \frac{L_i}{2}\|\xb-\cb^k_i\|^2.$$
 Combined with \cref{eq:grad_bound},
 $$
 \begin{array}{lll}
 F_i(\xb) - m_i(\xb;\cb^k_i) & = &  F_i(\xb) - F_i(\cb^k_i) - \gb_i^k(\cb^k_i;\delta_i)^\top(\xb-\cb^k_i)\\
 & \leq & 
 (\nabla F_i(\cb^k_i)-\gb(\cb^k_i;\delta_i))^\top(\xb-\cb^k_i) + \frac{L_i}{2}\|\xb-\cb^k_i\|^2.\\
 & \leq &
 L_i\left(\|\xb-\cb^k_i\|+\frac{\sqrt{n}\|\hat{V}_{Y_i}^{-1}(\xb;\cb^k_i)\|}{2\max\{\delta_i,\|\xb-\cb^k_i\|\}}\delta_i^2 \right)\|\xb-\cb^k_i\| + \frac{L_i}{2}\|\xb-\cb^k_i\|^2.
 \end{array}
 $$
 The theorem follows.
\end{proof}

\section{Additional Numerical Results}\label{sec:add_num}
\begin{figure}[h!]
 \centering
 \includegraphics[width=.99\textwidth]{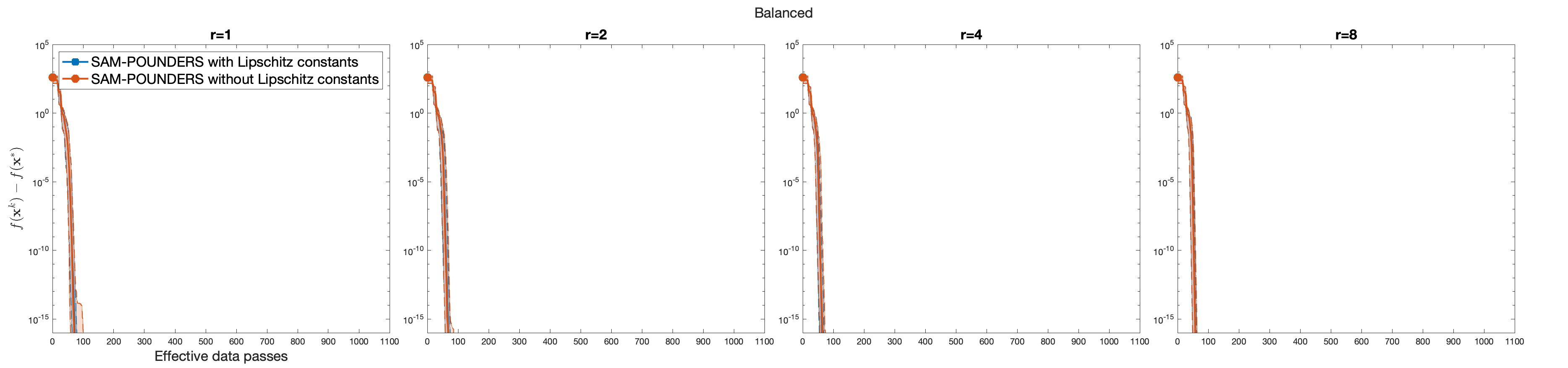}
 
 \includegraphics[width=.99\textwidth]{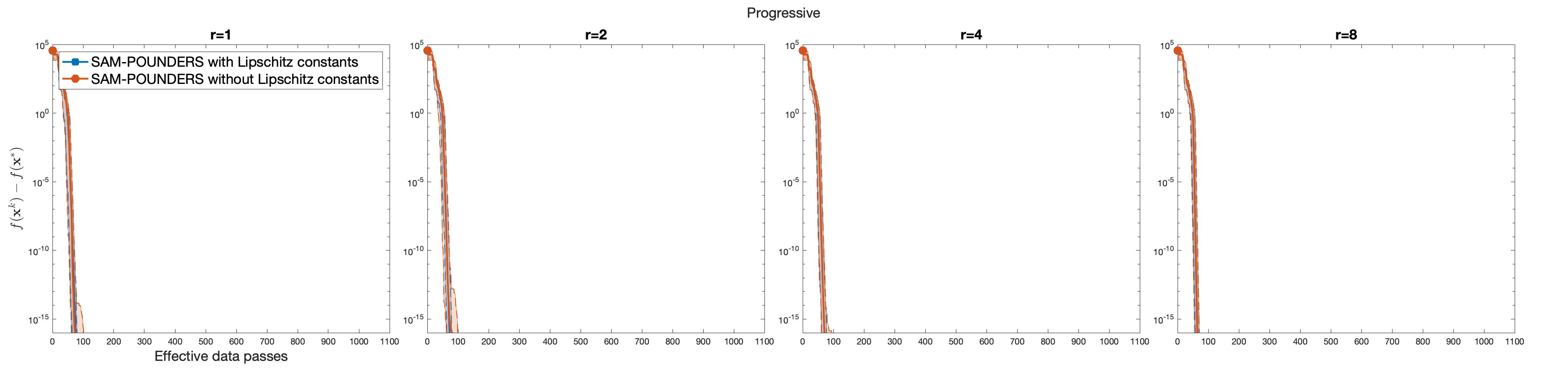}
 
 \includegraphics[width=.99\textwidth]{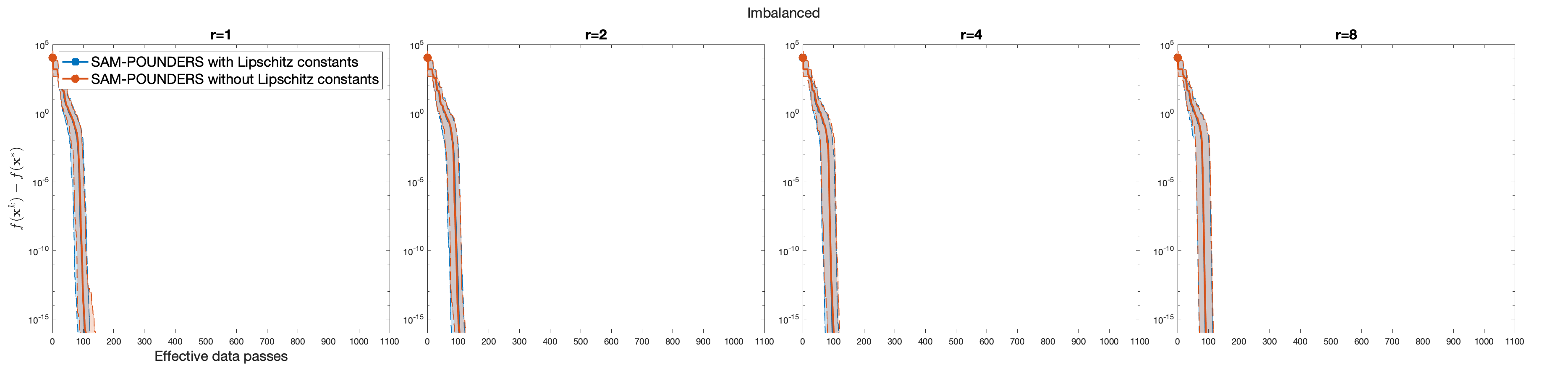}
 \caption{\label{fig:lip_nolip_rosenbrock} Comparing the performance of \texttt{SAM-POUNDERS} with dynamic batch selection and resource size $r$ with itself when provided with Lipschitz constants and when using our proposed scheme to dynamically adjust the Lipschitz constant estimates. These results show performance on the generalized Rosenbrock problems.}
\end{figure}

\begin{figure}[h!]
 \centering
 \includegraphics[width=.99\textwidth]{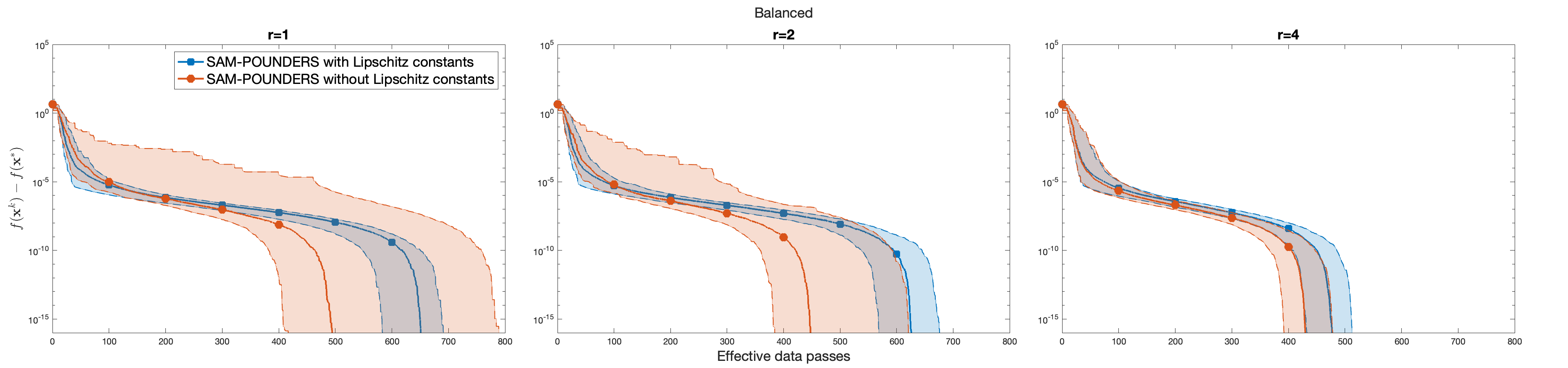}
 
 \includegraphics[width=.99\textwidth]{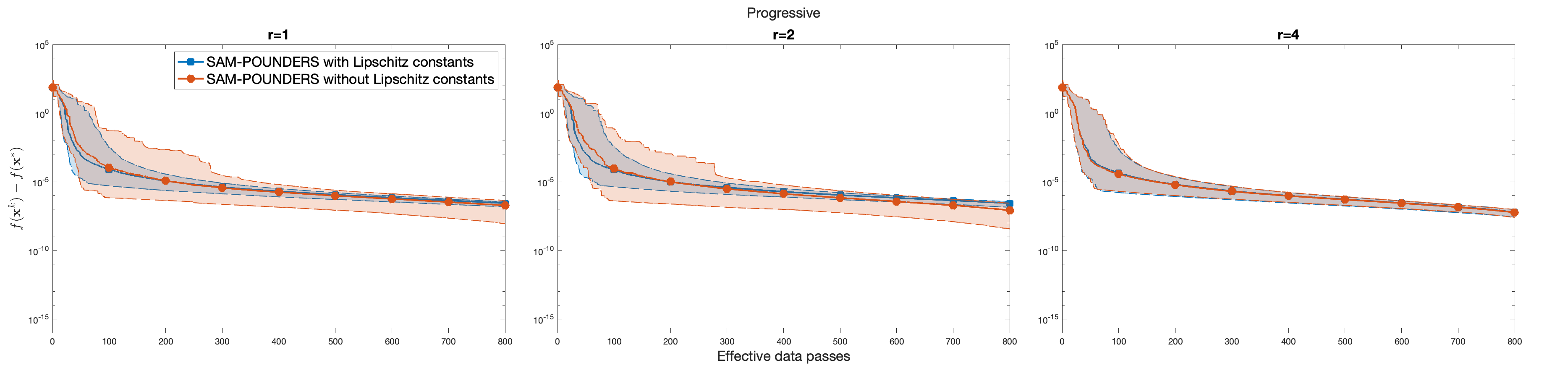}
 
 \includegraphics[width=.99\textwidth]{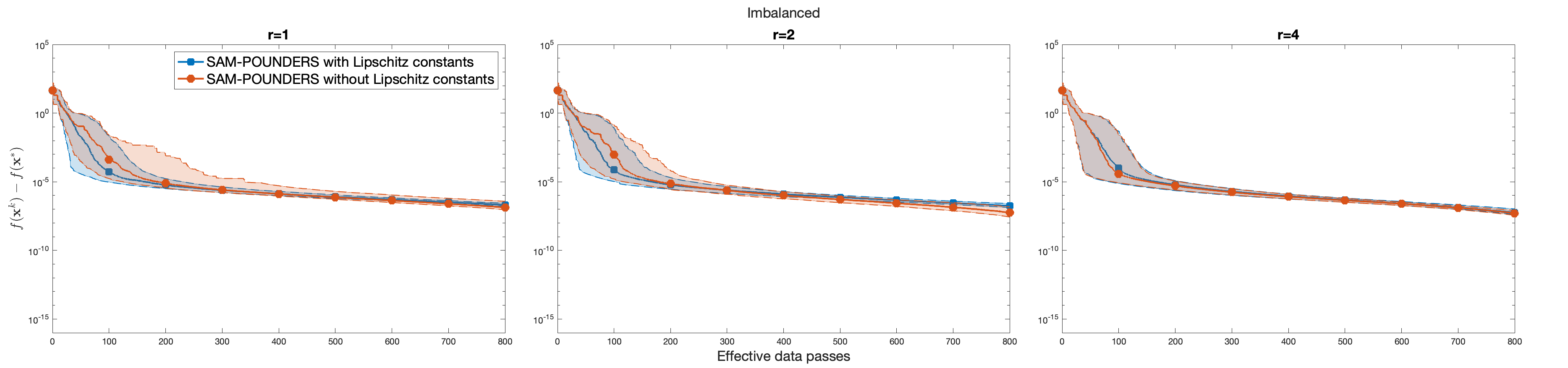}
 \caption{\label{fig:lip_nolip_cube} Comparing the performance of \texttt{SAM-POUNDERS} with dynamic batch selection and resource size $r$ with itself when provided with Lipschitz constants and when using our proposed scheme to dynamically adjust the Lipschitz constant estimates. These results show performance on the cube problems.}
\end{figure}

In both \Cref{fig:lip_nolip_rosenbrock} and \Cref{fig:lip_nolip_cube}, we intentionally leave the $x$-axis identical to the one employed in \Cref{fig:rosenbrock_compare} and \Cref{fig:cube_compare}, for easy visual comparison.
We see that while in all cases, some performance is lost when access to $L_i$ is taken away, the median performance of a dynamic variant without global Lipschitz constants is still always better than the median performance of the same method with uniform sampling. 

Our final results repeat the experiments in \Cref{sec:pounders_compare}, but with the larger dimensions $n=p=64$. 
We comment that many may see this as an atypical use of \texttt{POUNDERS}; many practical use cases of derivative-free optimization do not exceed more than a dozen variables (for instance, the 53 problems in the Mor{\'e}-Wild benchmarking set \cite{JJMSMW09} contains no problems larger than $n=12$). 
We also comment that in recent work by one author \cite{menickelly2023avoiding}, we have extended the ideas in this paper to randomly sampling in subspaces for higher dimensional derivative-free optimization; but even in that work,  we do not consider problems larger than $n=125$. 
Results are shown in \Cref{fig:pounders_rosenbrock_64} and \Cref{fig:pounders_cube_64}, respectively. 

\begin{figure}
    \caption{\label{fig:pounders_rosenbrock_64} Comparison of \texttt{POUNDERS} and \texttt{SAM-POUNDERS} on $n=p=64$-dimensional Rosenbrock problems with balanced (left), progressive (center), and imbalanced (right) modes of generation. Both solvers were given a common starting point, and \texttt{SAM-POUNDERS} was run with 30 different random seeds. Median performance is illustrated in the solid line of \texttt{SAM-POUNDERS}, with $10-90\%$ percentile bands illustrated in the transparent filled region.}
    \includegraphics[width=.99\textwidth]{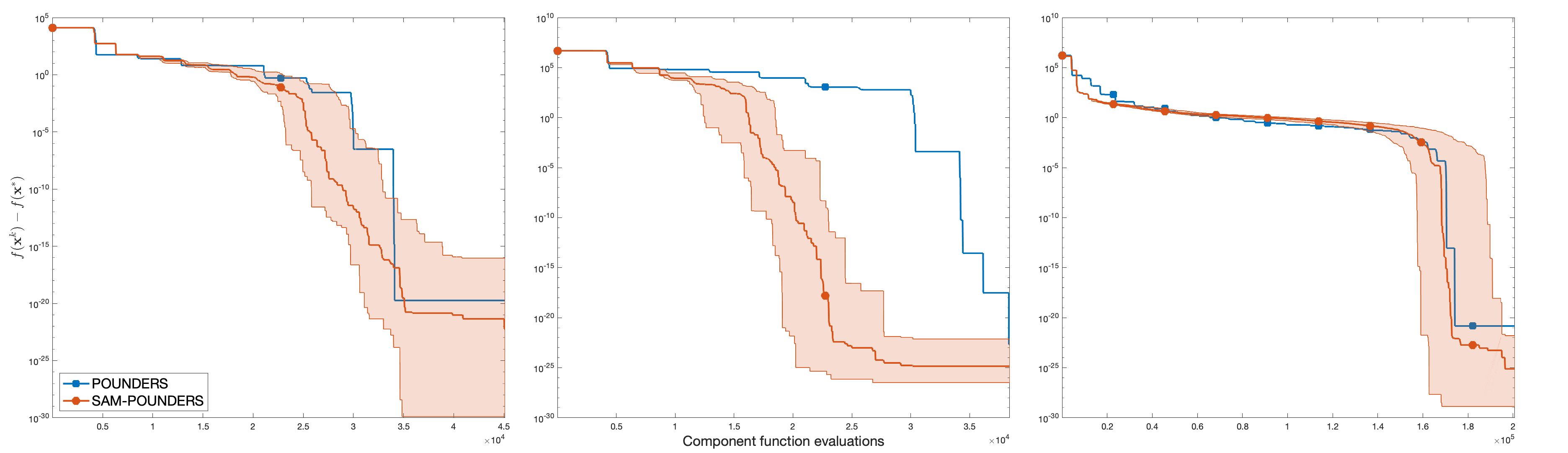}
\end{figure}

\begin{figure}
    \caption{\label{fig:pounders_cube_64} Same as \Cref{fig:pounders_rosenbrock_64}, but with cube problems}. 
    \includegraphics[width=.99\textwidth]{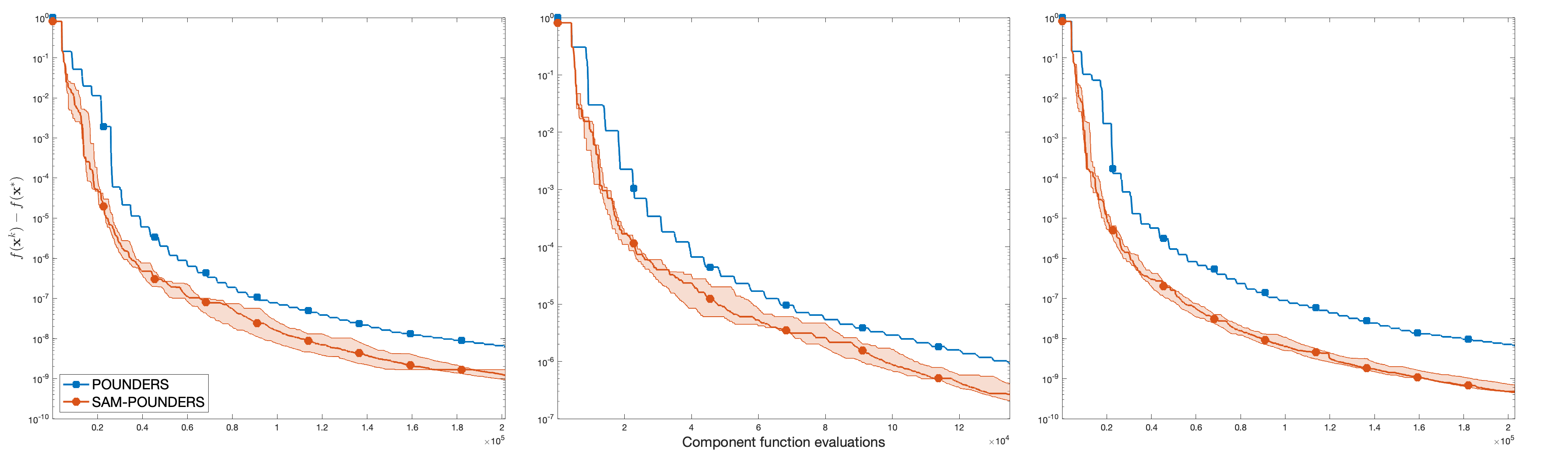}
\end{figure}

\end{appendix} 

\bibliography{smw-bigrefs.bib}
\bibliographystyle{abbrvnat}

\bigskip

\framebox{\parbox{.90\linewidth}{\scriptsize The submitted manuscript has been created by
        UChicago Argonne, LLC, Operator of Argonne National Laboratory (``Argonne'').
        Argonne, a U.S.\ Department of Energy Office of Science laboratory, is operated
        under Contract No.\ DE-AC02-06CH11357.  The U.S.\ Government retains for itself,
        and others acting on its behalf, a paid-up nonexclusive, irrevocable worldwide
        license in said article to reproduce, prepare derivative works, distribute
        copies to the public, and perform publicly and display publicly, by or on
        behalf of the Government.  The Department of Energy will provide public access
        to these results of federally sponsored research in accordance with the DOE
        Public Access Plan \url{http://energy.gov/downloads/doe-public-access-plan}.}}

\end{document}